\documentclass[final,1p,10pt,times]{elsarticle}
\usepackage[T1]{fontenc}
\usepackage{amsmath,amssymb,amsfonts,amsthm}
\usepackage{mathtools,array,booktabs,longtable}
\usepackage{microtype}
\usepackage{tikz}
\usetikzlibrary{arrows.meta,calc,patterns,decorations.pathreplacing}
\usepackage[hidelinks]{hyperref}
\allowdisplaybreaks[1]
\theoremstyle{plain}
\newtheorem{theorem}{Theorem}[section]
\newtheorem{corollary}[theorem]{Corollary}
\newtheorem{proposition}[theorem]{Proposition}
\newtheorem{lemma}{Lemma}
\numberwithin{equation}{section}
\makeatletter
\renewcommand\normalsize{\@setfontsize\normalsize{9.73pt}{10.4pt}%
	\abovedisplayskip 8pt plus2pt minus4pt
	\abovedisplayshortskip 0pt plus2pt
	\belowdisplayshortskip 5pt plus2pt minus3pt
	\belowdisplayskip\abovedisplayskip\let\@listi\@listI}
\normalsize\makeatother
\DeclareMathSizes{9.73}{9.4}{6.58}{4.7}
\begin{document}
	\begin{frontmatter}
		\title{Maximum length and size of simultaneous core partitions with distinct parts}
		\author{Metin Can Aydemir}
		\ead{can.aydemir@bilkent.edu.tr}
		\address{Department of Mathematics, Bilkent University, 06800 Ankara, T\"urkiye}
		\begin{abstract}
			We give, for the first time, a formula for the largest size of an $(s,t)$-core partition with distinct parts valid for every coprime pair. We also determine the greatest number of parts of an $(s,t)$-core partition whose successive parts differ by at least $D$, whenever $\gcd(s,t)\le D$, removing the restrictions on the core parameters in the formula of Burson, Sisneros-Thiry, and Straub. The size formula uses a column description and a single Dedekind sum. The length formula uses coefficients from the Euclidean algorithm and satisfies a recurrence in smaller core parameters. We construct partitions attaining each maximum.
		\end{abstract}
	\end{frontmatter}
	
	\section{Introduction}
	
	A partition is a finite weakly decreasing sequence $\lambda=(\lambda_1,\ldots,\lambda_\ell)$ of positive integers. Its size is $|\lambda|=\sum_i\lambda_i$, and its length is $\ell(\lambda)=\ell$. Its parts are $D$-distinct if $\lambda_i-\lambda_{i+1}\geq D$ for $1\leq i<\ell$, where $D$ is a positive integer; $D=1$ means distinct parts. The empty partition is included.
	
	The Young diagram of $\lambda$ has left-aligned rows of lengths $\lambda_i$. A cell's hook length counts itself and the cells to its right and below. A partition is an $s$-core if no hook length is divisible by $s$, and an $(s,t)$-core if it is both $s$-core and $t$-core.
	
	For the ordinary simultaneous core partitions, Anderson \cite{Anderson} enumerated coprime $(s,t)$-cores. Olsson and Stanton \cite[Theorem 4.1]{OlssonStanton} found the unique core of largest size. Its diagram contains every $(s,t)$-core diagram \cite[Theorem 2.4]{Vandehey}\cite[Theorem 5.1]{Fayers}, so it also has greatest length. Johnson \cite{Johnson} and Wang \cite{Wang} developed abacus and lattice-point descriptions and size statistics.
	
	Xiong \cite{XiongDistinct} and Straub \cite{Straub} proved Amdeberhan's \cite{Amdeberhan} Fibonacci conjecture for distinct-part $(s,s+1)$-cores; Xiong also found their largest size. Straub used nested twin-free sets for $(s,ms-1)$, and Nath and Sellers \cite{NathSellers} gave abacus constructions for $(s,ms\pm1)$. Sahin \cite{Sahin} studied $(s,s+1)$-cores with $D$-distinct parts; Kravitz \cite[Lemma 2.4]{Kravitz} proved finiteness exactly when $\gcd(s,t)\le D$.
	
	For length, Paramonov's refined polynomials \cite{Paramonov} treat $(s,ms+1)$, and Burson, Sisneros-Thiry, and Straub \cite[Lemma 23]{BSS} give the $D$-distinct maximum for $(s,ms+r)$ with $1\le r\le D$ and for $(s,ms-1)$. Pesikoff, Przybocki, and Xia \cite{PPX} determine the maximum hook length in every finite $D$-distinct family. For size with ordinary distinct parts, Xiong \cite[Theorems 1.1 and 1.3]{XiongSize} treats $(s,ms\pm1)$, and Yan, Qin, Jin, and Zhou \cite{Yan} treat odd $(s,s+2)$.
	
	We determine the greatest length for every finite $D$-distinct pair and the greatest size with distinct parts for every coprime pair. In the standard column description, a new bound on the triangular sum of run widths gives the length formula. Size also requires comparing run placements: a longest partition need not be largest.
	
	Write $\Lambda_D(s,t)$ for the greatest length. It is symmetric, vanishes if either parameter is $1$, is infinite if $\gcd(s,t)>D$, and equals $1$ if $D\ge\min(s,t)\ge2$; see Section \ref{sec:preliminaries}. Hence assume $2\le s<t$ and $\gcd(s,t)\le D<s$. We use $R=D+1$ because differences of at least $D$ between parts become differences of at least $R$ between beta-numbers. Put $x_+=\max\{x,0\}$, $T(z)=z(z+1)/2$, and let $\mathbf1_{\{E\}}$ be $1$ if $E$ holds and $0$ otherwise.
	
	Run the Euclidean algorithm on $t,s$, stopping at the first remainder $a$ with $1\leq a\leq D$, and let $g$ be the preceding remainder. Such a remainder occurs because the last positive remainder is $\gcd(s,t)\leq D$. Reversing these Euclidean steps gives
	\begin{equation}\label{eq:euclidean-matrix}
		s=dg+ua,\qquad t=eg+va,\qquad dv-eu=\varepsilon\in\{-1,1\}.
	\end{equation} Starting at $(1,0),(\lfloor t/s\rfloor,1)$, each Euclidean step $(g,a)\mapsto(a,g-ka)$ replaces a coefficient pair $(x,y)$ by $(kx+y,x)$. This fixes the coefficients, with $g>D$, $d,e,v\ge1$, $0\le u\le d$, and $u<d$ if $d\ge2$. If $u=0$, then $d=v=1$ and $g=s$; this occurs exactly when $t\bmod s\leq D$. Put
	\begin{equation}\label{eq:euclidean-parameters}
		g=qR+b,\qquad 0\leq b<R,\qquad \chi=\mathbf{1}_{\{u>0, a+b\geq R\}}.
	\end{equation} Thus $q\geq1$.
	
	\begin{theorem}\label{thm:recurrence}
		Let $s\geq2$ and $t=ms+r$, where $m\geq1$, $1\leq r<s$, and $\gcd(s,r)\leq D<s$. If $r>D$, put
		\begin{equation}\label{eq:A-definition}
			A_D(s,r)=qT(d)+\chi T(u),\qquad
			\kappa_D(s,r)=\mathbf1_{\{D\geq2,\ s\equiv D\bmod r,\ r\equiv 1\bmod R\}}.
		\end{equation} The parameters in $A_D(s,r)$ depend only on $s,r,D$, and
		\begin{equation}\label{eq:main-recurrence}
			\Lambda_D(s,ms+r)=mA_D(s,r)+\Lambda_D(r,s)-\kappa_D(s,r).
		\end{equation} If $1\leq r\leq D$, then
		\begin{equation}\label{eq:terminal}
			\Lambda_D(s,ms+r)=m\left\lfloor\frac{s}{R}\right\rfloor
			+\mathbf1_{\{\max(r,s\bmod R)\geq2\}}.
		\end{equation}
	\end{theorem}
	
	The recurrence ends at \eqref{eq:terminal}, the formula of \cite[Lemma 23]{BSS}, and its correction vanishes for ordinary distinct parts.
	
	\begin{theorem}\label{thm:closed}
		Let $2\leq s<t$ and $\gcd(s,t)\leq D<s$. With the notation in \eqref{eq:euclidean-matrix}-\eqref{eq:euclidean-parameters}, put
		\begin{equation}\label{eq:full-residual-length}
			K=\frac{de+d+e-1}{2},\qquad C=\frac{uv+u+v-1}{2}.
		\end{equation} Then
		\begin{equation}\label{eq:main-closed}
			\Lambda_D(s,t)=
			\begin{cases}
				qK+\mathbf1_{\{\max(a,b)\geq2\}},&\chi=0,\\
				qK+C+\mathbf1_{\{b\geq2\}},&\chi=1.
			\end{cases}
		\end{equation}
	\end{theorem}
	
	The determinant gives $\gcd(d,e)=1$ and, for $u>0$, $\gcd(u,v)=1$, so $K,C$ are integers; when $u=0$, $v=1$ and $C=0$.
	
	For the maximum size $M(s,t)$, restrict to distinct parts and coprime $2\le s<t$. We now use the unique positive integers $d,e$ with $|dt-se|=1$, $2d\le s$, $2e\le t$, and define
	\begin{equation}\label{eq:size-parameters}
		dt-se=\varepsilon\in\{-1,1\},\qquad 1\le d\le s/2, \qquad s=2dq+a,\quad 0\le a<2d,\quad u=(a-d)_+.
	\end{equation} Here $\varepsilon=1$ when $s=2$. Put
	\begin{equation}\label{eq:full-moments}
		h_j=\left\lceil\frac{je}{d}\right\rceil,\qquad K=\sum_{j=1}^dh_j=\frac{de+d+e-1}{2},\qquad J=\sum_{j=1}^djh_j,\qquad Z=\sum_{j=1}^dh_j^2.
	\end{equation} For $0\le k\le q$, define
	\begin{equation}\label{eq:size-full}
		F_k=\frac{k}{2}\left[2tJ-sZ+(s-\varepsilon)K-kK(K+2)\right].
	\end{equation} If $u>0$, let $$v=\frac{ue+\varepsilon}{d},\qquad g_j=\left\lceil\frac{jv}{u}\right\rceil,\qquad C=\sum_{j=1}^ug_j=\frac{uv+u+v-1}{2}.$$ Since $dv-ue=\varepsilon$, $v$ is a positive integer coprime to $u$. With $b_+=1$ and $b_-=2q+2$, put
	\begin{equation}\label{eq:size-residual}
		\begin{split}
			Q_{\rm sh}&=\sum_{j=1}^u\left((jt-b_\varepsilon)g_j-s\binom{g_j}{2}\right),\\
			F_{\rm res}&=F_q+Q_{\rm sh}-qK(C+1)-\binom C2.
		\end{split}
	\end{equation} Section \ref{sec:arithmetic} expresses $J,Z,Q_{\rm sh}$ through one Dedekind sum depending only on $s,d$.
	
	\begin{theorem}\label{thm:size}
		Put $W=2tJ-sZ+(s-\varepsilon)K$, $\mathcal D=K(K+2)$, and $k_* =\min\left\{q,\left\lfloor\frac{W+\mathcal D}{2\mathcal D}\right\rfloor\right\}$. The largest size of a distinct-part $(s,t)$-core is
		\begin{equation}\label{eq:main-size}
			M(s,t)=
			\begin{cases}
				F_{k_*},&d=1\text{ or }2,\\
				F_q,&d\ge3\text{ and }u=0,\\
				\max\{F_q,F_{\rm res}\},&d\ge3\text{ and }u>0.
			\end{cases}
		\end{equation} Moreover, $M(1,t)=0$ and $M(s,t)=M(t,s)$.
	\end{theorem}
	
	Sections \ref{sec:preliminaries}-\ref{sec:maximum} prove the length formulas, Sections \ref{sec:size}-\ref{sec:size-proof} the size formula, and Section \ref{sec:arithmetic} evaluates its coefficients by Dedekind reciprocity.

	\section{Columns and run bounds}\label{sec:preliminaries}
	\subsection{Beta-sets and ordered columns}
	
	The \emph{beta-set} $\beta(\lambda)=\{\lambda_i+\ell-i:1\le i\le\ell\}$ consists of the first-column hook lengths. Conversely, $B=\{b_1>\cdots>b_\ell>0\}$ gives parts $b_i-\ell+i$, including the empty partition when $B=\varnothing$. Thus
	\begin{equation}\label{eq:beta-statistics}
		\ell(\lambda)=|B|,\qquad |\lambda|=\sum_{b\in B}b-\binom{|B|}{2}.
	\end{equation} Here $\binom{x}{2}=x(x-1)/2$, also for real $x$. We use the following characterization from \cite[Lemmas 2.1 and 2.2]{Kravitz}.
	
	\begin{lemma}\label{lem:beta}
		A partition with beta-set $B$ is an $s$-core if and only if $b\in B$, $b\geq s$ imply $b-s\in B$. It has $D$-distinct parts if and only if any two different elements of $B$ differ by at least $R$.
	\end{lemma}
	
	The \emph{abacus} marks $B$ in residue columns modulo $s$. For an $s$-core, column zero is empty and every other column is filled from its bottom without gaps; its cardinality is its \emph{height}. Thus a $1$-core is empty. If $D\ge s$, Lemma \ref{lem:beta} forbids two beads in one column or two occupied columns, so $(1)$ attains greatest length one when $s,t>1$. If $h=\gcd(s,t)>D$, the beta-sets $\{1,1+h,\ldots,1+(k-1)h\}$ give cores of every length $k$. Henceforth $h\le D<s$.
	
	Put $p=s/h$. Addition by $t$ modulo $s$ gives $h$ cycles, one for each class modulo $h$, ordered by $$\rho_{ip+k}=i+kt-s\left\lfloor\frac{i+kt}{s}\right\rfloor\qquad(0\leq i<h,\ 0\leq k<p).$$ Let $x_j$ be the height at residue $\rho_j$, and let $j^-$ index its cyclic predecessor in the same cycle. Then
	\begin{equation}\label{eq:columns}
		B=\bigcup_{j=1}^{s-1}\{\rho_j+ks:0\leq k<x_j\},\qquad x_0=0,\quad x_j\in\mathbb Z_{\geq0}.
	\end{equation} Put $c_j=(t+\rho_{j^-}-\rho_j)/s$. This is a nonnegative integer, and subtraction of $t$ gives $(\rho_j+ks)-t=\rho_{j^-}+(k-c_j)s$. For $h=1$, $\rho_j=jt\bmod s$ and $c_j=\lfloor jt/s\rfloor-\lfloor(j-1)t/s\rfloor$, giving \cite[Proposition 4.1]{CHK} in this ordering.
	
	\begin{lemma}\label{lem:chain}
		The set \eqref{eq:columns} represents a $D$-distinct $(s,t)$-core if and only if
		\begin{equation}\label{eq:chain}
			0\leq x_j\leq x_{j^-}+c_j\qquad(1\leq j<s)
		\end{equation} and the following separation conditions hold. Let $S=\{j:x_j>0\}$ and let $U=\{\rho_j:j\in S\}$. Successive elements of $U$ in increasing order differ by at least $R$. If $|U|\geq2$, $\alpha=\min U$, $\beta=\max U$, and $s+\alpha-\beta\leq D$, then the column of residue $\alpha$ has height one.
	\end{lemma}
	\begin{proof}
		The subtraction formula gives \eqref{eq:chain}, equivalently \cite[Proposition 14]{Paramonov}, without a coprimality assumption. Within a column, differences are at least $s>D$. Between columns, a difference at most $D$ can occur only at the same or adjacent levels. Ordinary spacing excludes the former. The latter requires $s+\rho-\sigma\le D$ for occupied residues $\rho<\sigma$, necessarily $\rho=\alpha$, $\sigma=\beta$: moving either endpoint inward increases the difference by at least $R$. This conflict occurs exactly when column $\alpha$ has a second element. The characterization now follows from Lemma \ref{lem:beta}.
	\end{proof}
	
	Call $S\subseteq\{1,\ldots,s-1\}$ \emph{allowed} if its residues are $R$-separated in increasing order, with zero heights off $S$. The last condition of Lemma \ref{lem:chain} \emph{caps} column $\alpha$ at height one. A \emph{run} is a maximal string of consecutive selected columns within a cycle, \emph{full} if its length is $d$ and \emph{short} if smaller. For a cyclically $R$-separated residue set, every run has length at most $d$, since $dt\equiv\varepsilon a\pmod s$ with $1\le a\le D$. For a general allowed set, this bound applies after deleting the capped column.
	
	Every allowed set omits a column from each cycle, whose residues have spacing $h\le D$. Starting after an omitted column, set a capped height to one and use $x_j=x_{j^-}+c_j$ at other selected columns. Since $c_j\ge1$, these heights are positive, feasible, and maximal by \eqref{eq:chain}. This \emph{saturation} gives heights
	\begin{equation}\label{eq:run-heights}
		\left\lfloor\frac{jt+\rho}{s}\right\rfloor\qquad(1\le j\le\ell)
	\end{equation} on an uncapped run preceded by residue $\rho$. The length problem therefore reduces to choosing $S$.
	
	A residue set is \emph{cyclically $R$-separated} if every circular gap is at least $R$. Write $t=ms+r$, where $m\ge1$ and $1\le r<s$ because $\gcd(s,t)\le D<s$. The cycles and allowed sets depend only on $r$, and the increments $c_j$ grow by $m$, so a fixed support has saturated length $mA_S+B_S$. Delete any capped column to obtain a cyclically separated support $S'$. The coefficient of $m$ is zero off $S'$ and increases by one along each run, giving
	\begin{equation}\label{eq:slope}
		A_S=\sum_iT(\ell_i).
	\end{equation} Here $\ell_i$ are the run lengths of $S'$, including any split caused by deletion. Conversely, translate any cyclically $R$-separated set to avoid zero; its $r$-run lengths are unchanged and its support is uncapped.
	
	If $r>D$, setting $m=0$ gives a $D$-distinct $(s,r)$-core, possibly with zero selected heights. Indeed, at a capped residue $\alpha\le D-1<r$, the new increment is one, so the fixed height one is feasible; elsewhere the same recursion applies. Hence
	\begin{equation}\label{eq:general-upper}
		B_S\leq\Lambda_D(r,s),\qquad
		\Lambda_D(s,ms+r)\leq m\max_S A_S+\Lambda_D(r,s).
	\end{equation} We next determine $\max_S A_S$; the correction $\kappa_D(s,r)$ will require a sharper bound.
	
	\subsection{The run bound}\label{sec:run-bound}
	The parameters in \eqref{eq:euclidean-matrix}-\eqref{eq:euclidean-parameters} satisfy
	\begin{equation}\label{eq:euclidean-identities}dt-es=\varepsilon a,\qquad ut-vs=-\varepsilon g.\end{equation} Since $s\ge dg$ and $g>D\ge h$, we have $d<s/h$, so $P_d=\{0,r,\ldots,(d-1)r\}\pmod s$ has distinct points. Its clockwise gaps have lengths $g$ and $g+a$, with exactly $u$ of the latter.
	
	For $d\ge2$, let $f=e^{-1}\bmod d$, $1\le f<d$. For $0\le j<d$, $$jr\bmod s=\frac{(je\bmod d)s}{d}+\frac{\varepsilon ja}{d}.$$ The perturbation has magnitude less than $a<s/d$, so it does not change the order determined by $je\bmod d$. Thus the next orbit index is $j+f\bmod d$. If $\varepsilon=1$, then $f=d-u$; if $\varepsilon=-1$, then $f=u$. Subtracting the displayed coordinates gives gaps $g$ and $g+a$. In the positive case the long gaps start at indices $0,\ldots,u-1$; in the negative case they end at those indices.
	
	For $d=1$, the single gap $s=g+ua$ gives the same description. Thus $P_d$ is cyclically $R$-separated, whereas $dr\equiv\varepsilon a\pmod s$: $d$ is the first return within circular distance $D$ of zero, and bounds every cyclically separated run length.
	
	Put $w=\chi u$. The gap count yields
	\begin{equation}\label{eq:gapcapacity}
		\sum_{G\text{ gap of }P_d}\left\lfloor\frac G R\right\rfloor
		=(d-u)q+u\left\lfloor\frac{g+a}{R}\right\rfloor
		=qd+w.
	\end{equation}
	
	Column runs correspond to maximal strings $z,z+r,\ldots,z+(\ell-1)r$ modulo $s$, called cyclic $r$-runs.
	\begin{theorem}\label{thm:universal-run-bound}
		For every cyclically $R$-separated $U\subset\mathbb Z/s\mathbb Z$, its cyclic $r$-run lengths satisfy
		\begin{equation}\label{eq:energy}
			\sum_i T(\ell_i)\le qT(d)+T(w).
		\end{equation}
		Equality is attained.
	\end{theorem}
	We need two bounds on run lengths and an integer inequality.
	
	\begin{lemma}\label{lem:anchor}
		If the longest run has length $L$, then $$|U|\le qd+w+d-L.$$ \end{lemma}
	\begin{proof}
		By translation, we may take a longest run to be $P_L$. A gap $G_i$ admits at most $\lfloor G_i/R\rfloor-1$ further points, so $|U|\le\sum_i\lfloor G_i/R\rfloor$. We obtain $P_L$ by deleting $d-L$ points from $P_d$; each deletion merges two gaps and increases this sum by at most one. The bound therefore follows from \eqref{eq:gapcapacity}.
	\end{proof}
	
	\begin{lemma}\label{lem:longruns}
		If $d\ge2$, the sum of the lengths of any $q+1$ runs is at most $qd+w$. The same conclusion holds for a collection of fewer runs.
	\end{lemma}
	\begin{proof}
		Reflecting if necessary, we may assume $dr\equiv-a$ and $ur\equiv g\pmod s$. The lattice $\mathcal L=\langle(s,0),(r,1)\rangle$ then has basis $(-a,d),(g,u)$: both vectors belong to it and their determinant is $-s$. We represent a run $z,z+r,\ldots,z+(\ell-1)r$ by
		$$[z+(\ell-1)r,z+(\ell-1)r+R)\times[0,\ell)\quad\text{modulo }\mathcal L.$$
		At height $j+\tau$, $0<\tau<1$, subtracting $j(r,1)$ gives the interval of length $R$ starting at $z+(\ell-1-j)r$ on the circle of circumference $s$. Thus cyclic separation is equivalent to disjoint rectangle interiors.
		
		We first identify points differing by $(-a,d)$. The coordinate $\xi=x+ay/d$ is unchanged, while $y$ is periodic modulo $d$. Let $(c_i,\theta_i)$ be the coordinates of a rectangle's bottom-left corner. At an interior phase $y$, its horizontal interval has width $R$ and left endpoint $c_i+a\tau_i(y)/d$, where $0<\tau_i(y)<\ell_i$ and $\tau_i(y)\equiv y-\theta_i\pmod d$. Figure \ref{fig:sheared-packing} shows these sheared rectangles. The remaining period $(g,u)$ changes $(c_i,\theta_i)$ to $(c_i+c,\theta_i+u)$, where $c=g+au/d$.
		
		\begin{figure}[htbp]
			\centering
			\begin{tikzpicture}[x=1.3cm,y=0.68cm,font=\small,>=stealth]
				\filldraw[fill=black!8] (0,0) -- (2,0) -- (2.75,3) -- (0.75,3) -- cycle;
				\filldraw[fill=black!18] (2.6,1) -- (4.6,1) -- (5.35,4) -- (3.35,4) -- cycle;
				\node at (1.25,1) {$i$};
				\node at (4.25,3) {$i+1$};
				\draw[dashed,black!55] (-0.45,0) -- (2.6,0);
				\draw[dashed,black!55] (-0.45,1) -- (0.25,1);
				\node[anchor=east] at (-0.5,0) {$\theta_i$};
				\node[anchor=east] at (-0.5,1) {$\theta_i+\beta_i$};
				\draw[dashed,black!55] (-0.45,2) -- (5.1,2);
				\node[anchor=east] at (-0.5,2) {$y$};
				\draw[line width=1.2pt] (0.5,2) -- (2.5,2);
				\draw[line width=1.2pt] (2.85,2) -- (4.85,2);
				\draw[dotted] (0.5,2) -- (0.5,2.38);
				\draw[dotted] (2.85,2) -- (2.85,2.38);
				\draw[<->] (0.5,2.38) -- node[above] {$\gamma_i-a\beta_i/d$} (2.85,2.38);
				\draw[dotted] (2.6,0) -- (2.6,1);
				\draw[<->] (0,-0.55) -- node[below] {$\gamma_i=c_{i+1}-c_i$} (2.6,-0.55);
				\draw[<->] (3.35,4.42) -- node[above] {$R$} (5.35,4.42);
				\node[anchor=north] at (0,-0.02) {$c_i$};
				\node[anchor=north] at (2.6,-0.02) {$c_{i+1}$};
			\end{tikzpicture}
			\caption{Neighboring rectangles in $(\xi,y)$ coordinates, with $\beta_i>0$ shown. At a common height their left endpoints differ by $\gamma_i-a\beta_i/d$; both intervals have width $R$, giving \eqref{eq:centergap}. Across a phase boundary, first translate the second rectangle vertically by $d$.}
			\label{fig:sheared-packing}
		\end{figure}
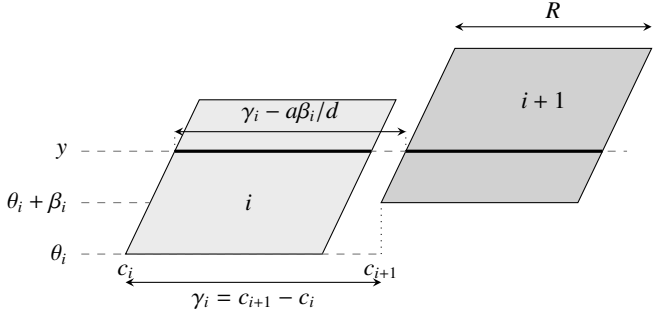
		
		Suppose $k=q+1$ runs have total length greater than $qd+w$. Put $m_i=d-\ell_i$ and $M=\sum_i m_i<d-w\le d$. Thus every pair has $\ell_i+\ell_j>d$ and overlapping vertical phases. We normalize the coordinates so that $c_i\in[0,c)$; they are distinct, since equal coordinates would give overlapping intervals whose left endpoints differ by less than $a<R$. Order them and put $c_{k+1}=c_1+c$, $\theta_{k+1}=\theta_1+u$, $\ell_{k+1}=\ell_1$, and $\gamma_i=c_{i+1}-c_i>0$.
		
		Choose $\beta_i\equiv\theta_{i+1}-\theta_i\pmod d$ in $[\ell_i-d,\ell_i)$. The corresponding vertical intervals overlap: if $\beta_i\ge0$ this follows from $\beta_i<\ell_i$; otherwise $\beta_i+\ell_{i+1}\ge\ell_i+\ell_{i+1}-d>0$. Their left endpoints differ by $\gamma_i-a\beta_i/d>-a>-R$, so disjointness forces
		\begin{equation}\label{eq:centergap}
			\gamma_i\ge R+\frac ad\beta_i.
		\end{equation}
		Since $\beta_i\ge\ell_i-d=-m_i$, we have $\sum_i\beta_i\ge-M$. Also $\sum_i\beta_i=u+nd$ for an integer $n$. Summing \eqref{eq:centergap} gives $kR\le g-an$. Since $kR>g$, we get $n\le-1$ and $M\ge d-u$, contradicting $M<d-w$ if $w=u$. If $w=0$, then $b+a<R$, so $kR>g+a$, giving $n\le-2$ and $M\ge2d-u>d$, again a contradiction. At most $q$ runs have total length at most $qd$, which also proves the assertion for fewer runs.
	\end{proof}
	
	\begin{lemma}\label{lem:integer}
		Let $q,d\ge1$, $0\le w<d$, and let $\ell_1\ge\ell_2\ge\cdots>0$ be integers with $L=\ell_1\le d$. Append zeros to the list when necessary. If $$\sum_i\ell_i\le qd+w+d-L,\qquad
		\sum_{i=1}^{q+1}\ell_i\le qd+w,$$ then $\sum_iT(\ell_i)\le qT(d)+T(w)$.
	\end{lemma}
	\begin{proof}
		We induct on $q$. If $L\le w$, then $2\sum_iT(\ell_i)\le(L+1)(qd+w+d-L)$, and the difference from $qd(d+1)+w(w+1)$ is
		$$ (d-L)(qd-L-1)+w(w-L)\ge0. $$
		Suppose $L>w$ and put $x=d-L$. For $q\ge2$, delete $L$ and put $w'=w+x<d$. Subtracting $L$ from the two bounds gives the hypotheses for $q-1,w'$, since the new largest part is at most $L$; the case of an empty remainder is immediate. Induction gives the result because $T(d)+T(w)-T(L)-T(w+x)=x(L-w)\ge0$.
		
		It remains to prove the case $q=1$. Put $h=L-w>0$. For a fixed total and upper bound on each part, moving units from smaller parts to larger ones increases the triangular sum. If $x<h$, the remaining parts are at most $w+x$ and total at most $w+2x$, so
		$$\sum_iT(\ell_i)\le T(L)+T(w+x)+T(x)\le T(d)+T(w),$$
		the last difference being $xh-T(x)\ge0$.
		
		If $x\ge h$, put $y=x-h$ and $z=2x-h$; the total bound is $2L+z$. For $z<L$, concentration gives $\sum_iT(\ell_i)\le2T(L)+T(z)$, and
		$$2[T(d)+T(w)-2T(L)-T(z)]
		=h(h-1)+y(y+1)+2y(L-1-z)\ge0.$$
		For $z\ge L$, we have $2y\ge w$ and $2\sum_iT(\ell_i)\le(L+1)(2d+w-L)$. The difference from $d(d+1)+w(w+1)$ is
		$$x(x-1)-wh=h(h-1)+y(y-1)+h(2y-w)\ge0.$$
	\end{proof}
	
	\begin{proof}[Proof of Theorem \ref{thm:universal-run-bound}]
		For $d\ge2$, the bound follows by combining Lemmas \ref{lem:anchor}, \ref{lem:longruns}, and \ref{lem:integer}. For $d=1$, all runs have length one and cyclic separation gives $|U|\le\lfloor s/R\rfloor=q+w$, with equality at $\{0,R,\ldots,(q+w-1)R\}$.
		
		For $d\ge2$, the construction in Lemma \ref{lem:positive-construction} uses only the orbit gaps above and has length coefficient $qT(d)+T(w)$ as $m$ varies. Saturation and the upper bound force equality.
		
	\end{proof}

	\section{Maximum length}\label{sec:maximum}
	
	Write $\Phi_D(s,t)$ for the right side of \eqref{eq:main-closed}, extended by symmetry. We prove attainment first, then the upper bound by recurrence.
	
	\subsection{Partitions attaining the formula}
	\begin{lemma}\label{lem:positive-construction}
		There is a $D$-distinct $(s,t)$-core of length equal to the right side of \eqref{eq:main-closed}.
	\end{lemma}
	\begin{proof}
		If $u=0$, then $d=v=1$, $g=s$, and $t=es+a$ with $1\le a\le D$, so \cite[Lemma 23]{BSS} applies. Assume $u>0$.
		
		Put $$F_k(c)=\{it-c-zs:1\le i\le k,\ z\in\mathbb Z_{\ge0},\ it-c-zs>0\},\qquad F_0(c)=\varnothing.$$ Take
		$$B=E\cup\bigcup_{j=0}^{q-1}F_d(c+jR)\cup F_{\chi u}(c_0),$$
		with these choices:
		$$\begin{array}{c@{\quad}l|ccc}
			&&c&c_0&E\\\hline
			\text{(i)}&\chi=0,\ a\ge2,\ \varepsilon=1&1&0&\varnothing\\[2mm]
			\text{(ii)}&\chi=0,\ a\ge2,\ \varepsilon=-1&-1&0&\{1\}\\[2mm]
			\text{(iii)}&\chi=0,\ a=1&c_\varepsilon&0&E_b\\[2mm]
			\text{(iv)}&\chi=1,\ \varepsilon=-1&1+\mathbf1_{\{b=1\}}&c+qR&\varnothing\\[2mm]
			\text{(v)}&\chi=1,\ \varepsilon=1,\ b=1&R+1&1&\varnothing\\[2mm]
			\text{(vi)}&\chi=1,\ \varepsilon=1,\ b\ge2&\max(R-1,a+1)&-1&\{1\}.
		\end{array}$$
		Here $c_+=2$, $c_-=1$, and $E_b=\{1\}$ if $b\ge2$, otherwise $\varnothing$. These cases exhaust $b\ge1$ when $\chi=1$, including $d=u=1$, where $\varepsilon=-1$ and the single gap is $g+a$.
		
		Put $A_i=it\bmod s$, $1\le i\le d$. In $P_d+t$, the long gaps follow $A_1,\ldots,A_u$ for $\varepsilon=1$ and precede them for $\varepsilon=-1$. The endpoints $A_d=\varepsilon a\bmod s$, $A_u=-\varepsilon g\bmod s$ enclose zero in a gap of length $g+a<2g$, so $(\min A_i,\max A_i)$ is $(a,s-g)$ or $(g,s-a)$, respectively.
		
		Subtraction of $s$ stays in a column, and subtraction of $t$ moves from $i$ to $i-1$, requiring the predecessor $1\in E$ only for shift $-1$. The residue bounds below exclude zero, proving both core conditions.
		
		The heights in $F_k(c)$ are $\lceil(it-c)/s\rceil$. Pairing fractional parts gives, for coprime $p,k>0$,
		\begin{equation}\label{eq:ceiling-block}\sum_{i=1}^{p}\left\lceil\frac{ik}{p}\right\rceil=\frac{pk+p+k-1}{2}.\end{equation}
		Hence $K=\sum_{i=1}^d\lceil ie/d\rceil$ and $C=\sum_{i=1}^u\lceil iv/u\rceil$. The perturbation in $t/s=e/d+\varepsilon a/(ds)$ crosses no integer for $i<d$, and adds one at $i=d$ exactly for $\varepsilon=1$. The orbit extrema then give cardinality $K$ for each full block, except $K+1$ for $F_d(1)$ in (i).
		
		Similarly, $t/s=v/u-\varepsilon g/(us)$ gives the residual cardinality $C+\mathbf1_{\{\varepsilon=-1,b\ge2\}}$. For the negative sign, the smallest residual residue is $g$ and all others are at least $2g$, so shifting by $c+qR$ removes one element precisely when $b=1$. For the positive sign, residual residues lie between $a+g$ and $s-g$, so shifting by $1$ or $-1$ does not change the cardinality.
		
		The translates $A_i-c-jR$, $0\le j<q$, form clusters with spacing $R$ and gaps $R+b$ or $R+b+a$. A residual point fits in a long gap at distance in $[R,b+a]$ from its left endpoint.
		
		The residual points lie $R$ before the next cluster in (iv), and $R$ after the preceding one in (v). In (vi) they lie $c+1$ after it, with $R\le c+1\le b+a$ because $c=\max(R-1,a+1)$, $b+a\ge R$, and $b\ge2$. Thus the blocks are cyclically separated.
		
		If $E=\{1\}$, the smallest other residue is at least $R+b+1$, $R+b-1$, or $a+b+R-c$ in (ii), (iii), or (vi), respectively. Each is at least $R+1$, since $b\ge2$ in (iii), (vi) and $c\le a+b-1$ in (vi). Adjoining $1$ therefore preserves ordinary spacing; its height one permits a cyclic conflict with the largest residue. Lemma \ref{lem:chain} gives $D$-distinctness, and the block counts give \eqref{eq:main-closed}.

	\end{proof}
	
	For $(s,t)=(11,15)$ and $D=1$, we have $(g,a)=(3,1)$, $(d,e)=(3,4)$, $(u,v)=(2,3)$, and $q=b=\chi=1$. Construction (v) gives $B=F_3(3)\cup F_2(1)$, with $K=9$, $C=5$, and $\Lambda_1(11,15)=14$; see Figure \ref{fig:full-residual-abacus}.
	
	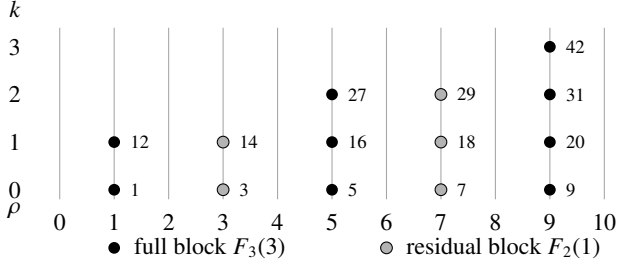
\begin{figure}[htbp]
		\centering
		\begin{tikzpicture}[x=0.72cm,y=0.63cm,font=\small]
			\foreach \r in {0,...,10}{
				\draw[black!35] (\r,-0.2) -- (\r,3.4);
				\node[anchor=north] at (\r,-0.32) {$\r$};
			}
			\foreach \k in {0,...,3}{
				\node[anchor=east] at (-0.55,\k) {$\k$};
			}
			\node[anchor=east] at (-0.55,3.8) {$k$};
			\node[anchor=east] at (-0.55,-0.42) {$\rho$};
			\foreach \r/\k/\b in {1/0/1,1/1/12,5/0/5,5/1/16,5/2/27,9/0/9,9/1/20,9/2/31,9/3/42}{
				\fill (\r,\k) circle (2.2pt);
				\node[anchor=west,font=\scriptsize] at (\r+0.13,\k) {$\b$};
			}
			\foreach \r/\k/\b in {3/0/3,3/1/14,7/0/7,7/1/18,7/2/29}{
				\filldraw[fill=black!30] (\r,\k) circle (2.2pt);
				\node[anchor=west,font=\scriptsize] at (\r+0.13,\k) {$\b$};
			}
			\fill (1,-1.22) circle (2.2pt);
			\node[anchor=west] at (1.2,-1.22) {full block $F_3(3)$};
			\filldraw[fill=black!30] (6,-1.22) circle (2.2pt);
			\node[anchor=west] at (6.2,-1.22) {residual block $F_2(1)$};
		\end{tikzpicture}
		\caption{The attaining beta-set at $(11,15)$. A bead at $(\rho,k)$ represents the labelled value $\rho+11k$. The full and residual blocks have respective heights $(2,3,4)$ and $(2,3)$.}
		\label{fig:full-residual-abacus}
	\end{figure}
	
	\subsection{The Euclidean recurrence}
	By \eqref{eq:slope} and Theorem \ref{thm:universal-run-bound}, the maximum slope in \eqref{eq:general-upper} is $A=qT(d)+\chi T(u)$.
	
	Put $e_0=e-md$, $v_0=v-mu$. Since $r=e_0g+v_0a$, swapping $(d,u),(e_0,v_0)$ gives the Euclidean matrix for $(r,s)$ with the same $g,a,q,b$, and $$K=mT(d)+\frac{d e_0+d+e_0-1}{2},\qquad
	C=mT(u)+\frac{uv_0+u+v_0-1}{2}.$$ For $v_0>0$, the residual branch and bonus are unchanged. For $v_0=0$, the determinant gives $e_0=u=1$, hence $g=r$, $a=s\bmod r$, and $s=dr+a$. The smaller pair has no residual branch, and its bonus differs by one exactly when $b=1$, $a=D\ge2$. Thus
	\begin{equation}\label{eq:formalrecurrence}
		\Phi_D(s,ms+r)=mA+\Phi_D(r,s)-\kappa,
		\quad
		\kappa=\mathbf1_{\{D\ge2,\ s\bmod r=D,\ r\bmod R=1\}}.
	\end{equation}
	
	If $\kappa=1$, then $r=qR+1$, $s=dr+D$, $D\ge2$, and \eqref{eq:terminal} gives $\Lambda_D(r,s)=qd+1$. If $B_0$ attains this smaller maximum, \cite[Lemma 6]{BSS} gives $B_0\subseteq\{1,\ldots,s-1\}$. There are at most $q$ occupied $r$-residues, each of height at most $d$, except for at most one residue below $D$, whose height may be $d+1$. Equality therefore forces $q$ residues $Y=\{y_1<\cdots<y_q\}$ and
	$$B_0=\bigcup_{j=0}^{d-1}(Y+jr)\ \cup\ \{y_1+dr\},\qquad y_1<D.$$
	The smallest column has height at least two, so $Y$ is cyclically $R$-separated. Its $q$ cyclic gaps sum to $qR+1$ and are therefore $R$ or $R+1$. Thus $B_0$ is maximal as an $R$-separated subset of $\{1,\ldots,s-1\}$: its internal gaps are less than $2R$, its first element is below $R$, and its final exterior gap is $D-1-y_1<R$.
	
	If an original support has intercept $qd+1$, its zero-parameter beta-set is such a $B_0$, contained in its support residues. Maximality forces the support residues to equal $B_0$, so there are no additional zero-intercept columns. In $r$-order it has $q-1$ ordinary $d$-runs and one capped $(d+1)$-run; any further merger would put two elements of $Y$ at cyclic distance $D<R$. Deleting the capped minimum leaves $q$ runs of length $d$, hence slope $qT(d)=A-1$ because $\chi=u=1$. Every support therefore loses at least one from its intercept or slope, and $m\ge1$ gives
	\begin{equation}\label{eq:critical}
		\Lambda_D(s,ms+r)\le mA+\Lambda_D(r,s)-1.
	\end{equation}
	
	\begin{proof}[Proof of Theorems \ref{thm:recurrence} and \ref{thm:closed}]
		We induct on the smaller core parameter, starting with \eqref{eq:terminal}. For $r>D$, we use \eqref{eq:general-upper} or, if $\kappa=1$, \eqref{eq:critical}. Together with the induction hypothesis and \eqref{eq:formalrecurrence}, this gives $$\Lambda_D(s,t)\leq mA_D(s,r)+\Phi_D(r,s)-\kappa_D(s,r)=\Phi_D(s,t).$$ Lemma \ref{lem:positive-construction} gives equality, proving \eqref{eq:main-closed}; then \eqref{eq:formalrecurrence} gives \eqref{eq:main-recurrence}. Section \ref{sec:preliminaries} covers the remaining parameter ranges.
	\end{proof}
	
	Write $\Lambda(s,t)=\Lambda_1(s,t)$. For ordinary distinct parts, the formula takes the following symmetric form.
	
	\begin{corollary}\label{cor:ordinary-distinct}
		Let $s,t\ge2$ be coprime. Choose the unique positive integers $d,e$ with $|td-se|=1$, $2d\le s$, $2e\le t$, and put $$h=\min\left\{\left\lfloor\frac{s}{2d}\right\rfloor,\left\lfloor\frac{t}{2e}\right\rfloor\right\},\qquad u=s-(2h+1)d,\quad v=t-(2h+1)e.$$ Then
		\begin{equation}\label{eq:ordinary-closed}
			\Lambda(s,t)=\begin{cases}
				\frac{h(d+1)(e+1)+(u+1)(v+1)}{2}-h-1&\text{if }u,v>0,\\
				\frac{h(d+1)(e+1)}{2}-h&\text{otherwise.}
			\end{cases}
		\end{equation}
	\end{corollary}
	\begin{proof}
		Assume $s<t$ and denote the Euclidean coefficients by $u_E,v_E$. Since $D=1$, we have $a=1$, $s=dg+u_E$, $t=eg+v_E$, and $g\ge2$, so $d,e$ are the stated coefficients. Moreover, $0\le u_E\le d$, $0<v_E\le e$, with at least one strict upper bound: if $u_E<d$, this follows from $v_E=(eu_E+\varepsilon)/d$; otherwise $d=1$, $\varepsilon=-1$, and $v_E=e-1$. Writing $g=2q+b$, $b\in\{0,1\}$, these bounds give $h=q$ and $u=u_E-(1-b)d$, $v=v_E-(1-b)e$. Both residual parameters are positive exactly when $b=1$ and $u_E>0$. The bonus indicators vanish; when $u_E=0$, we have $v_E=1$ and $C=0$. Substituting $K=((d+1)(e+1)-2)/2$ and $C=((u+1)(v+1)-2)/2$ in \eqref{eq:main-closed} gives the stated formula.
	\end{proof}
	
	For $s\ge2$, $m\ge1$, this recovers \cite{Paramonov}\cite[Lemma 23]{BSS}:
	\begin{equation}\label{eq:ordinary-closed-basic1}
		\Lambda(s,ms+1)=m\left\lfloor\frac{s}{2}\right\rfloor,\qquad
		\Lambda(s,ms-1)=m\left\lfloor\frac{s}{2}\right\rfloor-\mathbf1_{\{2\mid s\}}.
	\end{equation} The case $s=2$ follows directly from the height of its sole nonzero column.
	
	\section{Size candidates and bounds}\label{sec:size}
	
	Henceforth $D=1$, $2\le s<t$ are coprime, with the notation of \eqref{eq:size-parameters}-\eqref{eq:size-residual}.
	
	The columns form one cycle, $\rho_j=jt\bmod s$, with no caps. Since $dt\equiv\varepsilon\pmod s$, Lemma \ref{lem:chain} makes $S\subseteq\{1,\ldots,s-1\}$ allowed exactly when $S\cap(S+d)=\varnothing$ modulo $s$; thus every run has width at most $d$. Write $L=|B|$, $Q=\sum_{b\in B}b$, and $F=Q-\binom L2$.
	
	\begin{lemma}\label{lem:size-saturation}
		Every distinct-part $(s,t)$-core of maximum size has $x_j=0$ or $x_j=x_{j-1}+c_j$ at each column.
	\end{lemma}
	\begin{proof}
		If $0<x_j<x_{j-1}+c_j$, take the maximal block $j,\ldots,k$ of positive coordinates whose subsequent inequalities are tight. Adding $1$ to its heights is feasible by the initial slack; subtracting $1$ is feasible by positivity and the slack at any following positive coordinate. Neither change enlarges the support. For $b=k-j+1<s$, \eqref{eq:beta-statistics} makes the sum of the two new sizes exceed twice the old size by $sb-b^2>0$, contradicting maximality.
	\end{proof}
	
	In a saturated run preceded by $p$, its $j$th column contains
	\begin{equation}\label{eq:run-beta}
		\{jt+\rho_p-ks:1\leq k\leq\lfloor(jt+\rho_p)/s\rfloor\}.
	\end{equation} The length and beta-sum are additive over runs.
	
	\subsection{The attaining partitions}\label{sec:size-candidates}
	For $d\ge2$ and $0\le k\le q$, define
	\begin{equation}\label{eq:full-supports}
		S_k=\bigcup_{h=1}^{k}[p_h^\varepsilon+1,p_h^\varepsilon+d], \qquad p_h^+=s-2hd,\quad p_h^-=(2h-1)d.
	\end{equation} These integer intervals have cyclic gaps at least $d$, hence are allowed. By \eqref{eq:run-heights}, column $j$ of run $h$ has height $h_j$ and top $jt-2h+\frac{1-\varepsilon}{2}$. The resulting beta-set $B_k$ has cardinality $kK$; successive components are translates by $-2$, giving
	\begin{align}
		Q_1&=tJ-\frac{s}{2}(Z-K)-\frac{3+\varepsilon}{2}K, \label{eq:Q1-definition}\\
		Q_k&=kQ_1-k(k-1)K,\label{eq:Qk-definition}
	\end{align} Thus $F(B_k)=Q_k-\binom{kK}{2}=F_k$, as defined in \eqref{eq:size-full}.
	
	For $u>0$, put $r_h^\varepsilon=p_h^\varepsilon-\varepsilon d$ for $1\le h\le q$, $r_0^+=s-d$, and $r_0^-=d-u$, and define
	\begin{equation}\label{eq:residual-supports}
		S_{\rm res}= \bigcup_{h=1}^{q}[r_h^\varepsilon+1,r_h^\varepsilon+d] \ \cup\ [r_0^\varepsilon+1,r_0^\varepsilon+u].
	\end{equation} Again all gaps have at least $d$ zeros. The full components give $B_q-1$, while the short run has tops $jt-b_\varepsilon$. Put $c=2q+1$. Since $ut=sv-\varepsilon c$ and $0<b_\varepsilon+\varepsilon jc/u\le c+1<s/u$, the short-run heights are
	$$\left\lfloor\frac{jv}{u}+1-\frac{b_\varepsilon+\varepsilon jc/u}{s}\right\rfloor=\left\lceil\frac{jv}{u}\right\rceil=g_j\qquad(1\le j\le u).$$
	The resulting cardinality is $qK+C$ and the beta-sum is $Q_q-qK+Q_{\rm sh}$, giving \eqref{eq:size-residual}.
	
	For example, at $(s,t)=(10,13)$ the parameters are $(d,e,\varepsilon,q,u)=(3,4,-1,1,1)$ and $(K,J,Z,C)=(9,20,29,1)$. The candidates give $\Lambda(10,13)=10$ and $M(10,13)=115$:
	\begin{center}
		\begin{tabular}{lrr}
			\toprule
			Partition & Length & Size\\\midrule
			$(30,21,19,13,11,9,6,4,2)$ & $9$ & $115$\\
			$(28,19,17,11,9,7,6,5,3,1)$ & $10$ & $106$\\\bottomrule
		\end{tabular}
	\end{center}
	Thus the two maxima need not occur at the same partition. At $(11,15)$, the full moments are again $(K,J,Z)=(9,20,29)$, while $(u,v,C,Q_{\rm sh})=(2,3,5,71)$. This time
	$$F_q=136,\qquad F_{\rm res}=136+71-9(5+1)-\binom52=143,$$
	so the residual candidate, shown in Figure \ref{fig:full-residual-abacus}, is necessary.
	
	Let $A\subseteq B$ consist of the beta-numbers belonging to one run, and put $H=|A|$, $Q_A=\sum_{b\in A}b$, and $L=|B|$. Set $U=Q_A/H+(H+1)/2$. By \eqref{eq:beta-statistics}, deleting this run changes the size by
	$$F(B)-F(B\setminus A)=Q_A-\binom L2+\binom{L-H}{2}=H(U-L).$$
	We call this difference the \emph{contribution} of the run.
	
	Deleting a whole run preserves \eqref{eq:chain}, since no remaining positive column has a deleted positive predecessor. It also preserves distinctness. We may therefore choose, among the partitions of maximum size, one having the fewest parts. Every run of this partition has positive contribution: deleting a run with negative contribution would increase the size, while deleting one with zero contribution would preserve the size with fewer parts. Thus $L<U$ for every run of this chosen partition.
	
	The cases $d=1,2$ are treated separately at the beginning of Section \ref{sec:size-small}. For now assume $d\ge3$, so $q,a\ge1$, $e\ge d+1$, and $\gcd(d,e)=\gcd(s,d)=1$. By Lemma \ref{lem:size-saturation}, every partition of maximum size has a saturated support.
	
	\subsection{Run bounds}\label{sec:size-bounds}
	
	Put $T_2(x)=x(x+1)(2x+1)/6$. The errors $h_j-je/d$ permute $0,1/d,\ldots,(d-1)/d$; their squared sum gives the first identity below, and substitution gives the second:
	\begin{align}
		Z&=\frac{2e}{d}J-\frac{e^2T_2(d)}{d^2} +\frac{(d-1)(2d-1)}{6d},\label{eq:kernel-Z}\\
		\mathcal A&:=\frac{eJ-\frac d2(Z-K)}K=\frac{(K+1)(de+2K)}{6K}.\label{eq:size-scaled-full-sum}
	\end{align} Also
	\begin{equation}\label{eq:size-full-average-bounds}
		K>3(d-1),\qquad \frac{K+2}{2}<\mathcal A<\frac{2K}{3}.
	\end{equation} These follow from $K\ge d(d+3)/2>3(d-1)$ and $K+3\le de\le2K-6$ in the expression for $\mathcal A$.
	
	Put $\mu=Q_1/K$ and $w=\mu-2q+1$. Substituting $dt-se=\varepsilon$ in \eqref{eq:Q1-definition} and using $0<J/(dK)<1$ gives
	\begin{equation}\label{eq:size-weight-bounds}
		\frac{s}{d}\mathcal A-2q-1<w<\frac{s}{d}\mathcal A-2q.
	\end{equation}
	
	\begin{lemma}\label{lem:short-mean-elementary}
		\label{global-short-mean}\label{length-specific-actual-mean} A saturated run of $1\le\ell<d$ columns has cardinality $H$ and beta-sum $Q$ satisfying $$H\le\frac{e-1}{d}T(\ell)+\ell,\qquad \frac QH<M_\ell:=\frac sd\left[ \frac{(e-1)(2\ell+1)}6+\frac d2\right]<w.$$
	\end{lemma}
	\begin{proof}
		Let $p$ precede the run. Put $h=ep\bmod d$ for the positive sign, $h=(ep-1)\bmod d$ for the negative sign with $p>0$, and $h=0$ at the exceptional negative predecessor $p=0$. Write $r_j=(je+h)\bmod d$ and $H_j=(je+h-r_j)/d$. The $r_j$ are distinct and exclude $h$. Let $V$ be the sum of $je+h-kd$, $1\le k\le H_j$. Each actual beta-number differs from its scaled kernel value by $(p+j)/d$, $(s-p-j)/d$, or $-j/d$, respectively; hence $Q<(s/d)(V+H)$. With $T=T(\ell)$ and $T_2=T_2(\ell)$, $$dH=eT+\ell h-\sum r_j,\qquad 2dV=e^2T_2+e(2h-d)T+\ell(h^2-dh)+\sum(dr_j-r_j^2).$$ The height bound follows from $h\le d-1$ and $\sum r_j\ge\ell(\ell-1)/2$.
		
		Put $c=(e-1)(2\ell+1)/6+d/2-1$ and $\nu=\min(\ell,d-1-h)$. As a polynomial in the residues, $V-cH$ has $\partial(V-cH)/\partial r_j=(d-2r_j+2c)/(2d)>0$. Replacing the residues by the largest $\ell$ values excluding $h$, then summing and using $e\ge d+1$, gives
		$$2d(V-cH)\le-(\ell+1)\nu\left[\ell-\nu+\frac{(e-1)(\ell-1)}3\right]\le0,$$
		proving $Q/H<M_\ell$.
		
		Finally, $M_\ell\le M_{d-1}$ and $$\mathcal A-\left[\frac{(e-1)(2d-1)}6+\frac d2\right]=\frac e3+\frac{de}{6K}.$$ It follows from \eqref{eq:size-weight-bounds} that $w-M_{d-1}>se/(3d)-2q-1>0$, because $se-3d(2q+1)\ge2d[q(d-2)-1]+1>0$.
	\end{proof}
	
	\begin{theorem}\label{weighted-full-placement}\label{global-length-certificate}
		An allowed saturated support with $k$ full runs satisfies
		\begin{equation}\label{global-linear-certificate}
			Q-wL\le Kk(2q-k)\le q^2K.
		\end{equation} Equality in the first inequality holds only at $S_k$, including the empty support when $k=0$. If $L\le qK$, its size is at most $F_q$, with equality only at $S_q$. In particular, if $u=0$, then $M(s,t)=F_q$, attained uniquely at $S_q$.
	\end{theorem}
	\begin{proof}
		By Lemma \ref{global-short-mean}, discarding short runs increases $Q-wL$. We may therefore compare the remaining allowed full runs with $S_k$, without requiring compatibility with the discarded runs. For full-run predecessors $p_0<\cdots<p_{k-1}$, put $z_i=p_i-2di$. These indices are nondecreasing, with $0\le z_i\le s-(2k-1)d-1$.
		
		For $\varepsilon=1$, $d\rho_p=sh+p$ gives height $K-d+1+h$; the target $z_*=s-2kd$ has phase $h=d-1$ and the largest residue in that phase. For $\varepsilon=-1$ and $z_i>0$, $d\rho_p+p=sh$ gives height $K-d+h$; take $z_*=d$, of phase $h=d$. If this adds $1\le\Delta\le d-1$ entries, the residue gain $\delta$ satisfies $$d\delta=s\Delta+\varepsilon(z_*-z_i)\ge s\Delta-d+1.$$ Hence $s\Delta/d<\delta+1$, $\delta\ge2q\Delta$, and the old entries rise by $\delta$. By \eqref{eq:size-full-average-bounds}-\eqref{eq:size-weight-bounds}, the weighted gain exceeds
		\begin{align*}
			\delta(K-\Delta-\mathcal A)+2q\Delta-\mathcal A &\ge2q\Delta(K-\Delta-\mathcal A+1)-\mathcal A\\
			&>2(\Delta-1)(K/3-\Delta)\ge0.
		\end{align*} For $\Delta=0$, the gain is positive unless the run is already at the target; no intermediate placements are needed.
		
		For the negative sign, an initial block of $c\le k\le q$ zero indices may remain. Moving it to $d$ adds $jt-1$, $1\le j\le d$, to its first component and lowers each old entry by one. Its weighted gain is at least $tT(d)-qK-wd>0$, since
		\begin{align*}
			tT(d)-qK-wd &>s[K-\mathcal A-(d-1)/2]-qK+2qd-(d+1)/2\\
			&>q(d^2-2d+3)-(d+1)/2>0.
		\end{align*} Here $s\ge2qd$, $\mathcal A<2K/3$, and $K>3(d-1)$. At $S_k$, \eqref{eq:Qk-definition} gives \eqref{global-linear-certificate} and its equality condition.
		
		The bounds on $w,\mathcal A$ give $w>qK+\mathcal A/d-1>qK-1/2$. Thus, for $L\le qK$,
		$$F_q-F\ge(qK-L)(w+\tfrac12-qK)+\frac{(qK-L)^2}{2}\ge0,$$
		with equality only at $S_q$. For $u=0$, Corollary \ref{cor:ordinary-distinct} gives $\Lambda(s,t)=qK$, so the preceding bound proves the final assertion.
	\end{proof}
	
	\subsection{Complete residual profiles}\label{sec:size-placement}
	
	Assume $d\ge2$, $s=cd+u$, $c=2q+1$, and $0<u<d$. By \eqref{eq:run-beta}, a run of width $a$ preceded by $p$ is $I(at+\rho_p-s)$, where $I(z)=\{z-is-jt:i,j\in\mathbb Z_{\ge0},\ z-is-jt>0\}$. An allowed increase of every generator by an integer $b>0$ gives $B'=(B+b)\sqcup E$, $E\subseteq\{1,\ldots,b\}$. For $L=|B|>0$ and $m=|E|\le b$,
	\begin{equation}\label{eq:common-generator-shift}
		F(B')-F(B)=(b-m)L+\sum_{x\in E}x-\binom m2 \ge(b-m)L+m>0.
	\end{equation}
	
	Suppose the $q+1$ runs have widths $a_i$, all internal gaps have length $d$, and the gap at zero has length $d+\gamma$, where $0\le\gamma<u$. Put $$g=\sum_{i=1}^{q+1}(d-a_i)=d-u+\gamma<d.$$ Their predecessors in physical order are
	\begin{equation}\label{eq:actual-packed-predecessors}
		p_j=p+jd+\sum_{i=1}^ja_i\quad(0\le j\le q),\qquad0\le p<d+\gamma.
	\end{equation} The last occupied column is $s-d-\gamma+p$.
	
	\begin{lemma}\label{lem:complete-phase-normalization}
		If no allowed common increase of the generators is possible, then for $\varepsilon=1$ a short final run is preceded by $s-d$, or a full final run by $s-2d$; for $\varepsilon=-1$ the first run ends at $d$, or $\gamma>0$ and $p=d$. Every complete profile ($\gamma=0$) is dominated by one of these placements, allowing cyclic reordering of its widths.
	\end{lemma}
	\begin{proof}
		Put $\rho_*=\max_j\rho_{p_j}$. If $1\le b<s-\rho_*$ and residue $s-b$ is unoccupied, translate every residue by $b$ modulo $s$, equivalently every column index by $\varepsilon bd$. Cyclic separation is preserved and zero remains empty. No predecessor residue wraps, so the generator $at+\rho_{p_j}-s$ of each run of width $a$ increases by exactly $b$. Saturating the translated support therefore gives the common increase in \eqref{eq:common-generator-shift}.
		
		The residues $s-1,s-2$ cannot both be occupied. If $\rho_*\le s-3$, at least one of $b=1,2$ gives such an increase. Hence $\rho_*\ge s-2$. If $\rho_*=s-2$, the shift $b=1$ is blocked only when residue $s-1$, at column $-\varepsilon d$, is occupied.
		
		For $\varepsilon=1$, a predecessor of residue $s-1$ is $s-d$, necessarily preceding a short final run. Otherwise a predecessor of residue $s-2$ is $s-2d$, and column $s-d$ is occupied. If this run were short, $s-d$ would lie in its following gap. It is therefore full and final.
		
		For $\varepsilon=-1$, a predecessor of residue $s-1$ is $d$. It must be the first predecessor, since every later one exceeds $d$; thus $p=d$ and $\gamma>0$. Otherwise a predecessor of residue $s-2$ is $2d$. Since $p<d+\gamma<2d$ and $p_j\ge j(d+1)$, it must be the second predecessor. Its preceding gap has length $d$, so the first run ends at $d$.
		
		For $\gamma=0$, the shifts preserve every gap and the cyclic order. The last assertion follows by repeating these strictly improving shifts within the finite family.
	\end{proof}
	
	\begin{lemma}\label{lem:full-block-transport}
		In a saturated profile of $q+1$ runs and total width exceeding $qd$, two moves strictly increase size: move a full run by fewer than $d$ columns to a larger cardinality; or move a deficient full across $k\le q$ consecutive runs with gaps $d$ to cardinality $K$, rightward for $\varepsilon=1$ and leftward for $\varepsilon=-1$. Both moves must remain allowed inside $\{1,\ldots,s-1\}$, with positive crossed predecessors for the negative sign.
	\end{lemma}
	\begin{proof}
		For the second operation with positive sign, let $P$ precede the full run and write the crossed widths as $a_i=d-\delta_i$. Put $A_i=\sum_{j=1}^i\delta_j$, $A_0=0$, and $b=A_k<d$. The predecessors change as follows:
		$$\begin{array}{c|cc}
			&\text{before}&\text{after}\\\hline
			\text{moved full}&P&P+2kd-b\\
			\text{crossed run }i&P+2di-A_{i-1}&P+2d(i-1)-A_{i-1}.
		\end{array}$$
		For the negative sign, the corresponding move is obtained by reversing the permutation. The span and exterior gaps are unchanged, and internal gaps remain $d$, ensuring allowedness; by hypothesis no occupied index crosses zero.
		
		Crossed residues decrease by two. No beta-entry disappears: for the positive sign their indices exceed $2d$, whereas $1,2$ have indices $d,2d$. For the negative sign a full run and a gap of $d$ follow them, so they end before the indices $s-d,s-2d$ of $1,2$. Thus crossed cardinalities are unchanged and every crossed beta-entry decreases by two.
		
		Let the moved full have cardinality $K-\delta$, gain $\Delta$ entries, and raise each old entry by $\eta$. For the first move put $k=0$ and let $\xi$, $|\xi|<d$, be its predecessor displacement. The residue identities are
		$$\begin{aligned}
			d\eta&=s\Delta+\varepsilon\xi,&1\le\Delta\le\delta\le d&\quad\text{(first move)},\\
			d\eta&=s\delta+2kd-b,&\Delta=\delta\le d-1&\quad\text{(second move)}.
		\end{aligned}$$
		Since $s=cd+u$, integrality gives $\eta\ge c\Delta+2k$ in both cases. For crossed cardinality $H_\times\le kK$ and added set $E$,
		$$L'-L=\Delta,\qquad Q'-Q=\eta(K-\delta)-2H_\times+\sum_{x\in E}x.$$
		Since $|E|=\Delta$ and $\sum E\ge T(\Delta)$, subtracting the staircase gives
		$$\begin{aligned}
			F'-F&\ge\eta(K-\delta)-\Delta L-2kK+\Delta\\
			&>\Delta[q(K-2d)-d+1]\ge0,
		\end{aligned}$$
		Here $L\le qK+C<(q+1)K$ by Corollary \ref{cor:ordinary-distinct}, and $c\delta\Delta+2k\delta\le cd\Delta$ in both cases. Finally, $K\ge d(d+3)/2$ bounds the last bracket below by $(d-1)(d-2)/2$.
	\end{proof}
	
	\begin{proposition}\label{prop:residual-placement}
		Among supports with $q$ full runs and one run of width $u$, the largest size occurs at $S_{\rm res}$ for $\varepsilon=-1$, and at $S_{\rm res}$ or $S_B:=S_{\rm res}+d\pmod s$ for $\varepsilon=1$. In either case it is at most $\max(F_q,F_{\rm res})$.
	\end{proposition}
	\begin{proof}
		Every gap touches a full run and therefore has length at least $d$. Since the $q+1$ gaps have total length $s-(qd+u)=(q+1)d$, each has length $d$. By Lemma \ref{lem:complete-phase-normalization}, a short anchor gives $S_{\rm res}$, so it remains to consider $p=0$. Let $j$ full runs precede the short, whose predecessor is $2dj$.
		
		Write $F_B$ for the size at $S_B$. For $\varepsilon=1$ put $(F_*,L_*,k)=(F_B,qK+C-u,j)$; for $\varepsilon=-1$ put $(F_*,L_*,k)=(F_{\rm res},qK+C,q-j+1)$. The deficient fulls have heights $h_i-1$ for $i<d$ and $h_d$; at the negative predecessor zero the last height also drops by one, while the short gains one. Using $Q_k=kQ_1-k(k-1)K$ and summing \eqref{eq:run-beta} over these columns gives
		$$L=L_*-k(d-1),\qquad Q_*-Q=k[tT(d)-2L_*-d+k(d-1)],$$
		where $Q_*$ is the target beta-sum. Subtracting the staircase gives
		$$F_*-F=\frac{k(d+1)}2[dt-2L_*-1+k(d-1)]\ge0.$$
		Indeed, $C\le ue$ gives $dt-2(qK+C)\ge q(d-1)(e-1)+e(d-u)+\varepsilon>d$, so the bracketed quantity is positive. It remains to compare the positive placement $S_B$ with $S_q$ and $S_{\rm res}$. Here $ut=sv-c$; the short has heights $g_j-1$, cardinality $H=C-u=(u+1)(v-1)/2$, and beta-sum $Q_0=\sum_{j=1}^u[jt(g_j-1)-s\binom{g_j}{2}]$. Set $L_B=qK+H$. Passing from $S_B$ to $S_{\rm res}$ adds $u$ entries and changes the beta-sum by $tT(u)-(L_B+u)$, so
		\begin{align*}
			F_B-F_q&=H(U-L_B),&
			F_{\rm res}-F_B&=(u+1)[u(t-1)/2-L_B].
		\end{align*}
		If $u=1$, the first difference divided by $H$ is $[e+1-q((d-1)e+d+1)]/2<0$. For $u\ge2$, monotonicity of $g_j$ gives $\sum j(g_j-1)\ge(u+1)H/2$. The errors $jv/u-g_j+1$ permute $1/u,\ldots,1$; expanding their square sum in $Q_0$ gives
		$$\frac{u(t-1)}2-U\ge\frac{(v+1)[2s(u-1)-3u(u+1)]+6c}{12u}>0,$$
		since $s=cd+u\ge4u+3$ makes the bracket at least $5u(u-1)-6>0$. If $L_B\ge U$, then $F_B\le F_q$; otherwise $L_B<U<u(t-1)/2$ gives $F_B<F_{\rm res}$.
	\end{proof}

	\section{Proof of the maximum-size formula}\label{sec:size-proof}\label{sec:size-small}\label{sec:size-crowded}
	
	For $d=1$, we have $t=es+\varepsilon$ and $K=J=e$, $Z=e^2$. Thus \eqref{eq:size-full} simplifies to
	\begin{equation}\label{eq:size-distance-one}
		F_k=\frac{ek}{2}[(e+1)s+\varepsilon-(e+2)k].
	\end{equation}
	By \cite[Theorems 1.1 and 1.3]{XiongSize}, $M(s,t)=\max_{0\le k\le q}F_k$. For $\varepsilon=-1$, the two candidate families there are $F_k$ for $0\le k<s/2$ and $F_k+e(2k-s)\le F_k$ for $1\le k\le s/2$. They coincide at the only possible missing endpoint $k=s/2$. Rounding the vertex gives $k_*$ in Theorem \ref{thm:size}.
	
	\subsection{Inverse distance two}
	
	\begin{theorem}\label{d2-complete-theorem}
		If $d=2$, then $M(s,t)=\max_{0\le k\le q}F_k$, and every maximizing support is canonical. There are two maximizing partitions exactly when both integers nearest the vertex $W/(2\mathcal D)$ lie in $[0,q]$ and are equally near.
	\end{theorem}
	\begin{proof}
		Write $e=2\nu+1$, $\nu\ge1$. Then $s\ge5$ is odd, $2t=(2\nu+1)s+\varepsilon$, and $H=K=3\nu+2$ gives $$W=s(\nu+1)(5\nu+3)+\varepsilon(2\nu+1),\qquad \mathcal D=H(H+2).$$
		We call runs of widths one and two singletons and doubles, respectively. Every gap has at least two zeros, so $k$ doubles and $b$ singletons satisfy $4k+3b\le s$. Doubles have height at most $H$ and singletons have height $\nu$ or $\nu+1$, so $L\le sH/4$.
		
		For total singleton height $z$, we claim
		\begin{equation}\label{d2-global-bound}
			F\le F_k+\frac{z[s(\nu+1)-2kH-z]}2,
		\end{equation} with strict inequality unless the doubles are canonical. The full-run comparisons in Theorem \ref{weighted-full-placement} apply with weight $\omega=kH+z-1/2\le sH/4-1/2$. An ordinary height-increasing move adds one entry and raises old entries by at least $(s-1)/2$. Its weighted gain, and that of the exceptional negative move, are respectively at least
		$$\frac{(s-1)(H-1)}2+1-\omega\ge\frac{3\nu(s-2)}4+1>0,
		\qquad 3t-qH-2\omega\ge\frac{3\nu(2q+1)}2>0.$$
		For the doubles' cardinality $L_f$ and beta-sum $Q_f$, this gives $Q_f-\omega L_f\le Q_k-\omega kH$, with equality only at their canonical placement. Put $L_f=kH-\delta$. The singleton mean is at most $[s(\nu+1)-1]/2$. Subtracting $\binom{kH+z-\delta}{2}$ gives \eqref{d2-global-bound} minus $\delta^2/2$.Subtracting $\binom{kH+z-\delta}{2}$ gives an upper bound equal to the right-hand side of \eqref{d2-global-bound} minus $\delta^2/2$.
		Suppose $b\ge1$, $k<q$, and $F\ge F_k$. Then $s(\nu+1)-2kH>0$ by \eqref{d2-global-bound}. Extend $F_j=j(W-\mathcal D j)/2$ to real $j$, and put $h=s-4k-3b\ge0$ and $\theta=(\nu+1)(3\nu+1)$. Using $b\nu\le z\le b(\nu+1)$ and $\varepsilon\ge-1$, expansion gives
		$$2(F_{k+b/2}-F)\ge\frac b2[(b+h)\theta-e]\ge\frac{3\theta-e}2>\frac{\mathcal D}4.$$
		The first inequality drops the nonnegative terms $bk(3\nu^2-2)+b^2(7\nu^2-2\nu-4)/4$. The second uses $b(b+h)\ge b+2$, since $h\ge2$ if $b=1$, and $\theta>e$; the last reduces to $9\nu^2+2\nu-4>0$. Packing gives $k+\lceil b/2\rceil\le q$. The larger neighboring integer value is at least $F_{k+b/2}-\mathcal D/8$, hence exceeds $F$.
		
		For $k=q$ and $b>0$, packing forces $s=4q+3$, $b=1$. By Proposition \ref{prop:residual-placement}, it is enough to consider the residual placements. Substituting $u=1$ in their size formulas and using $q\ge1$ gives
		$$F_{\rm res}-F_q\le-2\nu-1<0,\qquad F_B-F_q\le-\nu<0,$$
		with $F_B$ present only for the positive sign. Thus only canonical doubles maximize size, and strict concavity of $F_k$ proves the remaining assertions.
		
	\end{proof}

	\subsection{Reducing the support}

	Henceforth $d\ge3$, $s=(2q+1)d+u$, $0<u<d$, and $L>qK$. Choose a size maximizer of minimum cardinality, so every run has $L<U$ by Section \ref{sec:size}. Here $s\ge10$ and $d<s/3$. Also $t\ge s+3$, since $t=s+1,s+2$ would give $d=1$ or $s\le2d+1$. Thus the next lemma applies to short runs of width at least two.

	\begin{lemma}\label{lem:local-extensions}
		Let $s,t$ be coprime, $s\ge10$, $t\ge s+3$, and $3\le n\le s/3$ an integer. A saturated run of $n-1$ columns with positive contribution can be strictly improved by extending it at either end, provided the new column is nonzero and empty and the resulting support is allowed.
	\end{lemma}
	\begin{proof}
		Let $\rho$ be the predecessor residue, $h_i=\lfloor(it+\rho)/s\rfloor$, $\tau=t/s$, $z=\rho/s$, and $c=\lceil(t-\rho)/s\rceil$. For an extension on the right, we append column $n$; for one on the left, we add $jt+\rho+ks$ for $0\le j<n$, $0\le k<c$. The respective gains are $h_n(G_R-L)$ and $cn(G_L-L)$, where
		$$G_R=nt+\rho-\frac{(s+1)h_n+s-1}{2},\qquad G_L=\frac{(n-1)t}{2}+\rho-\frac{s-1}{2}+\frac{(s-n)c}{2}.$$
		Both thresholds are at least $G=(s-1)(n\tau+z-1)/2$, since $h_n\le n\tau+z$ and $c\ge\tau-z$. The left additions preserve subtraction closure because $\rho+(c-1)s<t$; full saturation is unnecessary. As $L<U$, it remains to prove $U<G$.
		
		Write $w_i=i\tau+z$, $r_i=w_i-h_i$, and $\mathcal R=2H(G-U)$. The identity $$\mathcal R=H(s-1)(n\tau+z)-H^2-s\sum_{i=1}^{n-1}(w_i^2-r_i^2)$$ shows that it is symmetric in the errors $r_i$. On $[0,1]^{n-1}$,
		$$\frac{\partial\mathcal R}{\partial r_i}\le-n(s-n)+2s\le0.$$
		Coprimality makes the errors distinct multiples of $1/s$. In decreasing order they satisfy $r_i\le1-i/s$. Thus $\mathcal R\ge(n-1)P(\tau,z)$, where $P$ is obtained by setting $r_i=1-i/s$ in $\mathcal R/(n-1)$.
		
		This quadratic is concave in $z$, and for $\tau\ge1$,
		$$\frac{\partial P}{\partial\tau}\ge\frac{n^2(n-3)}2+\frac{n^2(s-n)}{2s}>0.$$
		It remains to check $P_0=P(1+3/s,0)$ and $P_1=P(1+3/s,1-1/s)$. Put $x=n-3$, $y=s-3n$. These are nonnegative integers with $x+y\ge1$. The polynomials $s^2P_0-72$ and $s^2P_1-75(x+y)+72$ have nonnegative coefficients in $x,y$, so
		$$s^2P_0\ge72,\qquad s^2P_1\ge75(x+y)-72\ge3.$$ Thus $\mathcal R>0$, proving $U<G$.
	\end{proof}
	
	\begin{lemma}\label{lem:actual-crowded-structure}
		A size maximizer of minimum cardinality with $L>qK$ has exactly $q+1$ runs, all cyclic gaps have length at least $d$, and its total occupied width lies between $qd+1$ and $qd+u$.
	\end{lemma}
	\begin{proof}
		Write $n=\sum\ell_i$, $\ell_{\min}=\min\ell_i$, and $\beta=(e-1)/d\ge1$. Since every run satisfies $H\le K\ell/d$, we have $n>qd$. For a short run, Lemma \ref{lem:short-mean-elementary} gives $L<U<\widehat U(\ell):=M_\ell+(\beta T(\ell)+\ell+1)/2$. A singleton satisfies $$U-qK\le2d+\tfrac12-qd(d-1)/2\le\tfrac12<L-qK,$$ unless $q=1$ and $d\in\{3,4\}$. In these cases write $e=dm+j$, $t=ms+r$, with $m\ge1$, $1\le j<d$, and $r=(sj+\varepsilon)/d$. Its one-column heights give
		$$U\le\frac{m(s+1)}2+\max\left\{r,\frac{s-1}{2}\right\}.$$
		Substitution, using $j\in\{1,d-1\}$, bounds this by $K+2$ for $d=3$ and $K+3-2m\le K+1$ for $d=4$. It remains to consider $d=3$, $L=K+1$. If the singleton has cardinality $H$ and there are $k\le1$ full runs, retaining its negative contribution in Theorem \ref{weighted-full-placement} gives
		$$F-F_1<w-(2-k)K-(w-M_1)H\le M_1-(2-k)K.$$
		For $k=0$ the last bound is negative. For $k=1,u=1$, it equals $(7-e)/3\le1$, so integrality gives $F\le F_1$, contradicting minimum cardinality. For $u=2$, it is $(16-e)/6\le2$, whereas \eqref{eq:size-residual} gives $F_{\rm res}-F_1=m(5m+2)-3(1-\varepsilon)/2\ge4$. Hence every run has width at least two.
		For the other short runs, we minimize $qK+1-\widehat U(\ell)$ at $u=d-1$ and then at the smallest $q,e$; the coefficients give the following bounds: $$
		\begin{array}{c|c|c}
			q&\ell&\text{lower bound for }qK+1-\widehat U(\ell)\\\hline
			q\ge2&d/3&[11(d-3)^2+25(d-3)+18]/36\\
			q=1&d/4&[29(d-8)^2+220(d-8)+128]/192.
		\end{array}$$ As $\widehat U$ increases, these bounds force $\ell>d/3$ for $q\ge2$ and $\ell>d/4$ for $q=1$; the second row is needed only for $d\ge8$, since $\ell\ge2$.
		
		For the local blocking argument, suppose a gap of $b<d$ zeros separates $[A,A+x-1]$ and $[A+x+b,A+x+b+y-1]$, with indices read cyclically. Their column differences fill $[b+1,x+y+b-1]$. Since $d$ is forbidden and $b+1\le d$,
		$$x+y+b\le d,$$
		so both runs are short. Both widths are at least two, so Lemma \ref{lem:local-extensions} applies to an extension into this gap.
		
		A nonzero new column $A+x$ must be blocked at distance $d$. A blocker at $A+x+d$ lies between the forbidden translates $[A+d,A+x+d-1]$ and $[A+x+b+d,A+x+b+y+d-1]$. Its run lies in the intervening $b$ positions; translation by $-d$ treats the other blocker. Since $x+y+b\le d$ and $s>3d$, this is a third run, of width
		$$z\le b,\qquad x+y+z\le d.$$
		The other endpoint similarly supplies a third run unless the gap is $\{0\}$.
		
		For the gap $\{0\}$, the neighboring indices $-1,1$ are occupied. If $d$ were occupied, its two neighbors would be forbidden by $-1,1$, making it a singleton; the same holds for $-d$. Thus $\pm d$ are empty. The index $1-\varepsilon d$ is empty because it differs from $1$ by $d$. Under multiplication by $t$, these empty indices show that the residues of $s-1$ and $t-1$ are unoccupied. Consequently $s-1,t-1\notin B$, and $B+1$ is still closed under subtraction of $s,t$: a new failure could only occur at $s$ or $t$. Translation preserves distinctness and raises size by $L>0$, a contradiction.
		
		Thus every gap is at least $\ell_{\min}$, and a width $a>d-2\ell_{\min}$ cannot touch a short gap, since its third run would give $a+2\ell_{\min}\le d$. If $q\ge2$, all widths exceed $d/3$, so that inequality forbids short gaps altogether. Writing $r$ for the number of runs, $rd\ge n>qd$ gives $r\ge q+1$, while $s\ge n+rd$ and $s<(2q+2)d$ exclude $r\ge q+2$. Hence $r=q+1$.
		
		Let $q=1$ and put $c=s/d\in(3,4)$. If every width were at most $d/2$, then $d\ge4$ and $Q/L\le M_{d/2}$ and, using $K\ge d(d+3)/2$, $$M_{d/2}+\tfrac12=\frac{c(2K+d)}6+\tfrac12\le\frac{2cK}{5},\qquad F\le\frac{2c^2K^2}{25}<\frac{(c-1)K^2}{2}<F_1.$$ Here the first inequality reduces to $cd(d-2)\ge15$, the middle one to $25(c-1)-4c^2\ge11$, and the last follows from $\mu>c(K+2)/2-2$. Hence some width $a>d/2$ has both adjacent gaps at least $d$, since $\ell_{\min}>d/4$. Three runs are impossible: a short gap gives $n\le d$, and otherwise $n\le u<d$. If there are at least four runs, take three successive runs other than the run of width $a$, of widths $x,y,z$ and intervening gaps $b_1,b_2$. Their span is at least $5\ell_{\min}>d$. The differences between columns in the first and third run fill the interval $$[b_1+y+b_2+1,\ x+b_1+y+b_2+z-1].$$ Since $d$ is forbidden, $b_1+y+b_2\ge d$. Consequently $$s\ge a+2d+x+y+z+b_1+b_2\ge a+x+z+3d>4d>s,$$ a contradiction. Thus there are two runs, and a short gap would require a third. Finally, $n>qd$ and the $q+1$ gaps give $n\le qd+u$.
	\end{proof}

	\begin{proposition}\label{prop:actual-completion}
		A minimum-cardinality maximizer with $L>qK$ has every gap exactly $d$.
	\end{proposition}
	\begin{proof}
		An incomplete profile has total excess gap length $0<\gamma<u$ by Lemma \ref{lem:actual-crowded-structure}. 
		
		Any internal gap longer than $d$ either adjoins a short, which extends by Lemma \ref{lem:local-extensions}, or separates full runs with predecessors $p$ and $p+2d+x$, where $0<x\le\gamma<d$. Their cardinalities differ, since $ex\not\equiv0\pmod d$; the negative predecessor zero has strictly smaller cardinality than every positive predecessor. We move the smaller full run by $x$ towards the other. The intervening gap becomes $d$, the other gap grows, and the full gains cardinality, contradicting Lemma \ref{lem:full-block-transport}. Thus all excess lies at zero and \eqref{eq:actual-packed-predecessors} applies.
		
		For $\varepsilon=1$, Lemma \ref{lem:complete-phase-normalization} leaves a short anchor at $s-d$ or a final full preceded by $s-2d$. The short anchor must have width $d-1$, else it extends right; then $p=d+\gamma-1$. In the full-anchor case $p=\gamma$. In both cases a first short would extend left, so the first run is full. Its height is below $K$, since its phase would reach $K$ only at $p\equiv u\pmod d$, whereas $\gamma,\gamma-1<u$. We move it across the $q-1$ intervening runs to predecessor $s-3d$ (short anchor) or $s-4d$ (full anchor). The target height is $K$; the crossed deficits are $g-1>0$ and $g>0$, respectively. By Lemma \ref{lem:full-block-transport}, this move increases the size, contradicting maximality. When $q=1$, these positive deficits would belong to an empty block, so the configurations are impossible.
		
		For $\varepsilon=-1$, any first short extends left, since it ends at $d$ or has $p=d$. If $p=0$, the first run is full and column $s-d$ is empty, so $1\notin B$. We replace $p=0$ by $p=d$: all old entries decrease by one and the $d$ entries $jt-1$ ($1\le j\le d$) are added. The gain is $$(d+1)[d(t-1)/2-L]>0,$$ since $L\le qK+C$ and the estimate in the proof of Proposition \ref{prop:residual-placement} gives $dt-2(qK+C)>d$. Thus only $p=d$ remains, with a first full of height $K$. A final full is deficient, since its predecessor $s-d-\gamma$ has nonzero phase $ew\pmod d$, where $w=u-\gamma$. Moving it left to predecessor $3d$ gives crossed deficit $g>0$ and increases the size by Lemma \ref{lem:full-block-transport}, contradicting maximality. For $q=1$, a positive crossed deficit is again impossible. A final short extends right unless $\gamma=1$ and it ends at $s-1$.
		
		In that last case, write its width as $n-1$ and put $h=\lfloor nt/s\rfloor$. Its column tops are $sh-it$, $1\le i<n$. Set $\rho=s-nt+sh\in(0,s)$ and reverse the columns: $sh-it=(n-i)t+\rho-s$. The bound $U<G$ in the proof of Lemma \ref{lem:local-extensions} therefore gives $U<(s-1)h/2$. Its hypotheses were checked above, with $3\le n\le d$. Adding $s,2s,\ldots,hs$ and lowering all entries by one is allowed: column $s-d$ was empty, the final short gains a column, and the excess gap disappears. The size gain $(h+1)[h(s-1)/2-L]$ is positive, again contradicting maximality.
	\end{proof}
	
	\subsection{Comparing complete profiles}
	\begin{proposition}\label{prop:actual-complete-comparison}
		A size maximizer of minimum cardinality with $L>qK$ has size at most $\max(F_q,F_{\rm res})$.
	\end{proposition}
	\begin{proof}
		Proposition \ref{prop:actual-completion} and Lemma \ref{lem:complete-phase-normalization} give a complete, normalized profile. Write its widths as $a_i=d-\delta_i$ $(0\le i\le q)$, where $\sum\delta_i=g=d-u$, and put $$Y=\sum_{i=1}^q\delta_i,\quad a_0=u+Y,\quad E_2=\sum_{i<j}\delta_i\delta_j.$$ If $E_2=0$, the desired bound follows from Proposition \ref{prop:residual-placement}. Otherwise every $a_i>u\ge1$, and we show that $F<F_{\rm res}$.
		
		Put $r_z=d\lceil ez/d\rceil-ez$, so $0\le r_z<d$ and $r_{z+d}=r_z$. Number the non-anchor runs in physical order for $\varepsilon=1$ and in reverse order for $\varepsilon=-1$. Put $y_h=\sum_{i=h}^q\delta_i$, $y_{q+1}=0$, $c=2q+1$, and $\kappa=\mathbf1_{\{Y=g\}}$. Summing widths and gaps from the anchor, last for $\varepsilon=1$ and first for $\varepsilon=-1$, gives the physical predecessors $P_i$ below. The auxiliary triples below use offsets $p_i$, distinct from the predecessors $P_i$.
		$$\begin{array}{cc|c|ccc}
			\varepsilon&i&P_i&p_i&a_i&b_i\\\hline
			+&0&s-d-d\kappa&0&a_0&1\\
			+&h&u+2d(h-1)+y_h-d\kappa&y_h&a_h&c+2-2h\\
			-&0&g-Y&-Y&a_0&c\\
			-&h&2d(q+1-h)-y_{h+1}&-y_{h+1}&a_h&c-2h
		\end{array}\qquad(1\le h\le q).$$
		For data $(p,a,b)$, column $j$, $1\le j\le a$, has height and top
		$$h_j=\frac{ej+r_{p+j}-r_p}{d},\qquad T_j=jt+\frac{\varepsilon p-sr_p}{d}-b.$$
		Indeed, for the positive sign $eu\equiv-1\pmod d$ gives $d\rho_{P_h}=s(d-1-r_{y_h})+P_h$, hence $T_j=jt+\rho_{P_h}-s+\kappa$. For the negative sign, $d\rho_{P_h}=s(d-r_{-y_{h+1}})-P_h$ gives $T_j=jt+\rho_{P_h}-s+1$. At the anchor, $\rho_{P_0}-s=-1-\kappa$ in the positive sign and $(Y-sr_{-Y})/d-(c+1)$ in the negative sign, also when $P_0=0$. The height formula now follows from \eqref{eq:run-heights}.
		
		For the negative sign the table describes $(B+1)\cup\{0\}$, with zero at the bottom of the last anchor column because $p+a_0=u$ and $r_u=d-1$. For the positive sign it describes $B+1$ when $Y=g$, in which case we adjoin zero separately, and $B$ otherwise. Whenever padding is used, the cardinality and beta-sum become $L+1$ and $Q+L$, with unchanged size $Q+L-\binom{L+1}{2}=F$. In the following calculations, this zero is counted exactly once.
		
		We set each endpoint error $r_p$ to zero, retain the $r_{p+j}$, and define $F_0$ by the size polynomial at these possibly fractional heights. Expanding the column beta-sum $h_jT_j-s\binom{h_j}{2}$ gives
		\begin{equation}\label{eq:interval-moments}
			2d^2Q_j^0=se^2j^2+e[sd-2bd+2\varepsilon(p+j)]j+\Psi(p+j,b),
		\end{equation}
		where $\Psi(z,b)=(sd-2bd+2\varepsilon z)r_z-sr_z^2$. Put $\alpha=(1+\varepsilon)/2$ and
		$$\mathcal E=\sum_{h=1}^q\sum_{j=1}^{y_h-1+\alpha}r_{\varepsilon j}\ge0.$$ Empty sums are zero.
		To compare with the residual candidate, whose deficits are $(g,0,\ldots,0)$, we define $G_0$ by
		\begin{equation}\label{eq:zero-error-comparison}
			F_{\rm res}-F_0=G_0+\frac{\mathcal E}{d}.
		\end{equation}
		
		The arguments $p+j$ range over
		$$\begin{array}{c|cc}
			&\text{anchor}&\text{run }h\\\hline
			+&[1,u+Y]&[y_h+1,d+y_{h+1}]\\
			-&[1-Y,u]&[1-y_{h+1},d-y_h].
		\end{array}$$
		Let $\Delta_r$ denote profile minus canonical for sums over these intervals. The identities
		$$\Psi(z+d,b)=\Psi(z,b)+2\varepsilon dr_z,\qquad
		\Psi(z,b+2)=\Psi(z,b)-4dr_z$$
		make the endpoint terms telescope. For the positive sign, deleting $[1,y_h]$ and adding $d+[1,y_{h+1}]$ to each full period gives the non-anchor difference $-\sum_{j=1}^Y\Psi(j,c+1)-2d\mathcal E$. The anchor adds $\sum_{j=1}^Y\Psi(u+j,1)$. For the negative sign, adding $[1-y_{h+1},0]$ and deleting $d+[1-y_h,0]$ gives $-\sum_{j=0}^{Y-1}\Psi(-j,c)-2d\mathcal E$, whose sum cancels the anchor's addition.
		
		For the positive sign, $ue\equiv-1\pmod d$ gives
		$$\begin{aligned}
			r_{u+j}&=r_j+1-d\mathbf1_{\{j=g\}}\qquad(1\le j\le g),\\
			\Psi(u+j,1)-\Psi(j,c+1)&=s(d-1)-2d+2u+2j.
		\end{aligned}$$
		At the only wrap, $j=g$, we have $r_g=d-1$ and both sides of the second identity equal $s(d-1)$. Thus in either sign
		$$\Delta_r\sum\Psi=\alpha Y[s(d-1)-2d+2u+Y+1]-2d\mathcal E.$$
		
		Write $L_c,Q_c$ for the padded canonical quantities and $L_0,Q_0$ for the zero-error quantities. Then $L_c=qK+C+1-\alpha$. Expanding $a_i=d-\delta_i$ gives
		$$qT(d)+T(u)-\sum_iT(a_i)=E_2,\qquad
		\Delta S_2:=qT_2(d)+T_2(u)-\sum_iT_2(a_i)\ge(d+u+1)E_2,$$
		since the remainder in the second relation is $\sum_{i<j<k}\delta_i\delta_j\delta_k\ge0$.
		The same intervals give $\Delta_r\sum r_z=\alpha(Y-d\kappa)$. Counting the extra zero when $\varepsilon=1$ and $Y=g$ cancels the $\kappa$ term, so
		$$\Delta L:=L_c-L_0=\frac{eE_2-\alpha Y}{d}.$$
		Define $\Delta B_1=\sum_{\rm canonical}bT(a)-\sum_{\rm profile}bT(a)$ and $\Delta P_1=\sum_{\rm canonical}pT(a)-\sum_{\rm profile}pT(a)$. Summing \eqref{eq:interval-moments} yields
		\begin{align*}
			2d^2\left(Q_c-Q_0-\frac{\mathcal E}{d}\right)
			={}&se^2\Delta S_2+e[sdE_2-2d\Delta B_1+2\varepsilon\Delta P_1+2\varepsilon\Delta S_2]\\
			&-\alpha Y[s(d-1)-2d+2u+Y+1].
		\end{align*}
		Since $se^2+2\varepsilon e>0$, we may replace $\Delta S_2$ by its lower bound. Subtracting $\binom{L_c}{2}-\binom{L_0}{2}=\Delta L(L_c-1/2)-(\Delta L)^2/2$ and discarding the nonnegative term $E_2^2e^2$ gives
		$$2d^2G_0\ge(q+1)d(d+2u+1)E_2e^2+C_1e-\alpha(q+1)d(d-1)Y,$$
		where
		$$\begin{aligned}
			C_1={}&[(q+1)d(d+1)+\varepsilon(3d+u+1)]E_2\\
			&-2d\Delta B_1+2\varepsilon\Delta P_1+2\alpha Y\sum_iT(a_i).
		\end{aligned}$$
		Put $A_0=T(a_0)-T(u)=Y(2u+Y+1)/2$; the non-anchor triangular deficits sum to $E_2+A_0$. The triples, $y_h\le Y$ and $y_{q+1}=0$ give
		$$\begin{aligned}
			\Delta B_1&\le(2q-1+2\alpha)E_2+2\alpha qA_0,\\
			\varepsilon\Delta P_1+\alpha Y\sum_iT(a_i)&\ge-(1-\alpha)qT(d)Y.
		\end{aligned}$$
		Since $2A_0\le Y(d+u+1)$, substitution in $C_1$ gives
		$$C_1+2qd(d+u+1)Y\ge[qd(d-3)+d^2-u-1]E_2\ge0.$$
		Also $E_2\ge g-1$ and $Y\le g\le2E_2$, so
		$$2d^2G_0\ge(q+1)dE_2[(d+2u+1)e(e-4)+4ue-2(d-1)]>0,$$
		since $e\ge d+1\ge4$ and $u\ge1$.
		
		To restore the endpoint errors, consider an original run of width $a\ge2$ and put $n=a+1$. If $q=1$, then $E_2>0$ makes both deficits positive, so $n\le d<s/3$; if $q\ge2$, then $n\le d+1\le s/3$. The proof of Lemma \ref{lem:local-extensions} therefore gives $U<G$. With $\rho$ its predecessor residue, put $V=t(a+1)/2+\rho-(s-1)/2$. Since $V-G=[(a+1)t+(s+1)\rho]/(2s)>0$, we have $V>U>L$. Padding increases every top and $L$ by one, preserving $V-L$; we use $V_i,L$ for the corresponding table quantities.
		
		Setting $r_{p_i}=0$ increases every height in run $i$ by $x_i=r_{p_i}/d$ and every column top by $sx_i$. The resulting size polynomial $F_0$ therefore satisfies
		$$F_0-F=\sum_i a_ix_i(V_i-L)+\frac{s\sum_i a_ix_i^2-(\sum_i a_ix_i)^2}{2}\ge0,$$
		by Cauchy-Schwarz and $\sum_i a_i<s$. Therefore $F_{\rm res}-F\ge F_{\rm res}-F_0=G_0+\mathcal E/d>0$.
		
	\end{proof}
	
	\begin{proof}[Proof of Theorem \ref{thm:size}]
		The cases $d\le2$ follow from the $d=1$ calculation and Theorem \ref{d2-complete-theorem}. For $d\ge3$, the required upper bound follows from Theorem \ref{global-length-certificate} when $u=0$ or $L\le qK$, and from Proposition \ref{prop:actual-complete-comparison} otherwise. Since the supports in Subsection \ref{sec:size-candidates} attain the candidate values, the result follows.
	\end{proof}
	
	For odd $s\ge3$ and $t=s+2$, we have $d=(s-1)/2$, $e=d+1$, $\varepsilon=-1$, $q=1$, and $u=0$. Thus $M(s,t)=F_1$, also for $d=1,2$ by the preceding cases. Substituting $h_j=j+1$ in \eqref{eq:full-moments}-\eqref{eq:size-full} recovers \cite[Theorem 3.10]{Yan}:
	\begin{equation}\label{eq:size-gap-two}
		M(s,s+2)=\frac{(s^2-1)(s+3)(5s+17)}{384}.
	\end{equation}
	
	\section{Arithmetic evaluation}\label{sec:arithmetic}
	
	For $a\in\mathbb Z$, $b\in\mathbb Z_{>0}$ and $\gcd(a,b)=1$, the classical Dedekind sum is $$\mathfrak s(a,b)=\sum_{j=1}^{b-1}\left(\frac jb-\frac12\right)\left(\left\{\frac{aj}{b}\right\}-\frac12\right),$$ where $\{x\}=x-\lfloor x\rfloor$ is the fractional part and $\mathfrak s(a,1)=0$. We use $\mathfrak s(a,b)=\mathfrak s(a\bmod b,b)=\mathfrak s(a^{-1},b)$ and $\mathfrak s(-a,b)=-\mathfrak s(a,b)$, with the inverse taken modulo $b>1$. Dedekind reciprocity \cite[Theorem 2.2]{BeckRobins} gives
	\begin{equation}\label{eq:dedekind-reciprocity}
		\mathfrak s(a,b)+\mathfrak s(b,a)=\frac{a^2+b^2+1}{12ab}-\frac14\qquad(a,b>0,\ \gcd(a,b)=1).
	\end{equation} Periodicity and \eqref{eq:dedekind-reciprocity} give a Euclidean recursion ending at $\mathfrak s(a,1)=0$.

	\begin{proposition}\label{prop:arithmetic-evaluation}
		Let $\sigma=\mathfrak s(s,d)$. For positive integers $a,b$, put
		$$\mathcal P(a,b)=\frac{(a^2-1)s(b^2+1)}{12a}+\frac{b(a+1)[3s+(a+2)t]}{12}.$$
		Then
		\begin{align}
			J&=\frac{e(d+1)(2d+1)}6+\frac{d(d-1)}4+\varepsilon d\sigma,\label{eq:J-dedekind}\\
			F_k&=k[\mathcal P(d,e)+\sigma]-\frac{k^2K(K+2)}2.\label{eq:F-dedekind}
		\end{align} Together with \eqref{eq:kernel-Z}, the first identity evaluates both height sums. If $u>0$, put $c=2q+1$. Then
		\begin{equation}\label{eq:R-dedekind}
			\begin{split}
				F_{\rm res}={}&F_q+\mathcal P(u,v)+c\left[\frac{d^2+u^2+1}{12du}-\frac14-\sigma\right]\\
				&-qK(C+1)-\frac{C(C+c+1)}2.
			\end{split}
		\end{equation}
	\end{proposition}
	\begin{proof}
		For $d>1$, we expand the Dedekind sum using $h_j=je/d+1-(je\bmod d)/d$ for $j<d$. Since $e^{-1}\equiv-\varepsilon s\pmod d$, we have $\mathfrak s(e,d)=-\varepsilon\sigma$, giving \eqref{eq:J-dedekind}. Substituting this expression for $J$ in \eqref{eq:kernel-Z} and \eqref{eq:size-full} gives \eqref{eq:F-dedekind}. For $d=1$ these identities hold directly.
		
		For the residual term, we apply the same moment identities to $(u,v)$, using $ut-sv=-\varepsilon c$ and $\mathfrak s(v,u)=\varepsilon\mathfrak s(d,u)$. Since $s\equiv u\pmod d$, reciprocity gives $\mathfrak s(d,u)=(d^2+u^2+1)/(12du)-1/4-\sigma$. Substitution in \eqref{eq:size-residual} gives \eqref{eq:R-dedekind}.
	\end{proof}

	For fixed $s,r$ and $t=ms+r$, the parameters $d,\varepsilon,q,u,\sigma$ are fixed and $e,v$ are affine in $m$. Thus each $F_k$ and $F_{\rm res}$ are quadratic in $m$, and for $d\ge3$, $M(s,ms+r)$ is one quadratic or the maximum of two for every $m\ge1$.

	\section*{Acknowledgments}
	The author is grateful to his PhD advisor, Professor Hamza Ye\c{s}ilyurt, and to Professor Mustafa \c{C}elebi P{\i}nar for their guidance and helpful feedback.


\begin{thebibliography}{99}
		
		\bibitem{Amdeberhan}
		T. Amdeberhan, \emph{Theorems, problems and conjectures}, arXiv:1207.4045v7 (2022).
		
		\bibitem{Anderson}
		J. Anderson, \emph{Partitions which are simultaneously $t_1$- and $t_2$-core}, Discrete Math. \textbf{248} (2002), 237-243.
		
		\bibitem{BeckRobins}
		M. Beck and S. Robins, \emph{Dedekind sums: a combinatorial-geometric viewpoint}, in Unusual Applications of Number Theory, M. B. Nathanson (ed.), DIMACS Ser. Discrete Math. Theoret. Comput. Sci., vol. 64, American Mathematical Society, 2004, pp. 25-35.
		
		\bibitem{BSS}
		H. E. Burson, S. Sisneros-Thiry, and A. Straub, \emph{Refined counting of core partitions into $d$-distinct parts}, Electron. J. Combin. \textbf{28} (2021), P1.37.
		
		\bibitem{CHK}
		H. Constantin, B. Houston-Edwards, and N. Kaplan, \emph{Numerical sets, core partitions, and integer points in polytopes}, in Combinatorial and Additive Number Theory II, M. B. Nathanson (ed.), Springer Proc. Math. Stat., vol. 220, Springer, 2017, pp. 99-127.
		
		\bibitem{Fayers}
		M. Fayers, \emph{The $t$-core of an $s$-core}, J. Combin. Theory Ser. A \textbf{118} (2011), 1525-1539.
		
		\bibitem{Johnson}
		P. Johnson, \emph{Lattice points and simultaneous core partitions}, Electron. J. Combin. \textbf{25} (2018), P3.47.
		
		\bibitem{Kravitz}
		N. Kravitz, \emph{On the number of simultaneous core partitions with $d$-distinct parts}, Discrete Math. \textbf{342} (2019), no. 12, 111592.
		
		\bibitem{NathSellers}
		R. Nath and J. A. Sellers, \emph{Abaci structures of $(s,ms\pm1)$-core partitions}, Electron. J. Combin. \textbf{24} (2017), P1.5.
		
		\bibitem{OlssonStanton}
		J. B. Olsson and D. Stanton, \emph{Block inclusions and cores of partitions}, Aequationes Math. \textbf{74} (2007), 90-110.
		
		\bibitem{Paramonov}
		K. Paramonov, \emph{Cores with distinct parts and bigraded Fibonacci numbers}, Discrete Math. \textbf{341} (2018), 875-888.
		
		\bibitem{PPX}
		E. Pesikoff, B. Przybocki, and J. Xia, \emph{The maximum hook length of $d$-distinct simultaneous core partitions}, Electron. J. Combin. \textbf{30} (2023), P3.30.
		
		\bibitem{Sahin}
		M. Sahin, \emph{Core partitions with $d$-distinct parts}, Australas. J. Combin. \textbf{72} (2018), 19-28.
		
		\bibitem{Straub}
		A. Straub, \emph{Core partitions into distinct parts and an analog of Euler's theorem}, European J. Combin. \textbf{57} (2016), 40-49.
		
		\bibitem{Vandehey}
		J. Vandehey, \emph{Containment in $(s,t)$-core partitions}, arXiv:0809.2134 (2008).
		
		\bibitem{Wang}
		V. Y. Wang, \emph{Simultaneous core partitions: parameterizations and sums}, Electron. J. Combin. \textbf{23} (2016), P1.4.
		
		\bibitem{XiongDistinct}
		H. Xiong, \emph{Core partitions with distinct parts}, Electron. J. Combin. \textbf{25} (2018), P1.57.
		
		\bibitem{XiongSize}
		H. Xiong, \emph{On the largest sizes of certain simultaneous core partitions with distinct parts}, European J. Combin. \textbf{71} (2018), 33-42.
		
		\bibitem{Yan}
		S. H. F. Yan, G. Qin, Z. Jin, and R. D. P. Zhou, \emph{On $(2k+1,2k+3)$-core partitions with distinct parts}, Discrete Math. \textbf{340} (2017), 1191-1202.
		
	\end{thebibliography}
\end{document}